\documentclass[12pt,english]{extarticle}
\usepackage[utf8]{inputenc}
\usepackage[a4paper]{geometry}
\usepackage{fancyhdr}
\usepackage{color}
\usepackage{verbatim}
\usepackage{amsmath}
\usepackage{amsthm}
\usepackage{amssymb}

\makeatletter

\providecommand{\tabularnewline}{\\}

\numberwithin{equation}{section}
\numberwithin{figure}{section}
\theoremstyle{plain}
\newtheorem{thm}{\protect\theoremname}[section]
\theoremstyle{definition}
\newtheorem{defn}[thm]{\protect\definitionname}
\theoremstyle{remark}
\newtheorem{rem}[thm]{\protect\remarkname}
\theoremstyle{plain}
\newtheorem{lem}[thm]{\protect\lemmaname}
\theoremstyle{plain}
\newtheorem{prop}[thm]{\protect\propositionname}
\theoremstyle{plain}
\newtheorem{cor}[thm]{\protect\corollaryname}

\usepackage[english]{babel}
\usepackage{amsthm}
\usepackage{amsfonts}

\usepackage[yyyymmdd]{datetime}

\usepackage{color}

\definecolor{rot}{rgb}{1.000,0.000,0.000}

\usepackage{pst-node}

\newcommand{\N}{\mathbb{N}}
\newcommand{\R}{\mathbb{R}}
\newcommand{\eps}{\varepsilon}

\newcommand{\cE}{\mathcal{E}}
\newcommand{\cR}{\mathcal{R}}
\newcommand{\cD}{\mathcal{D}}

\newcommand{\e}{\mathrm{e}}

\newcommand{\fD}{\mathfrak D}

\renewcommand{\L}{\mathrm{L}}
\newcommand{\W}{\mathrm{W}}

\newcommand{\la}{\langle}
\newcommand{\ra}{\rangle}

\newcommand{\Gammlim}{\xrightarrow{\text{$\Gamma$}}}
\newcommand{\wGammlim}{\stackrel{\text{$\Gamma$}}{\rightharpoonup}}
\newcommand{\Moscolim}{\xrightarrow{\text{M}}}

\newcommand{\wstarlim}{\stackrel{*}{\rightharpoonup}}

\newcommand{\D}{\mathrm{D}}
\renewcommand{\d}{\mathrm{d}}
\renewcommand{\div}{\mathrm{div}}

\newcommand{\eff}{\mathrm{eff}}
\newcommand{\slow}{\mathrm{slow}}
\newcommand{\fast}{\mathrm{fast}}
\newcommand{\diff}{\mathrm{diff}}
\newcommand{\react}{\mathrm{react}}

\definecolor{rot}{rgb}{1.000,0.000,0.000}

\newcommand{\Prob}{\mathrm{Prob}}

\newcommand{\GamLimE}{\stackrel{\text{$\Gamma_\mathrm{E}$}}{\rightarrow}}
\newcommand{\wGamLimE}{\stackrel{\text{$\Gamma_\mathrm{E}$}}{\rightharpoonup}}
\newcommand{\MoscoLimE}{\stackrel{\text{$\mathrm{M}_\mathrm{E}$}}{\longrightarrow}}

\makeatother

\usepackage{babel}
\providecommand{\corollaryname}{Corollary}
\providecommand{\definitionname}{Definition}
\providecommand{\lemmaname}{Lemma}
\providecommand{\propositionname}{Proposition}
\providecommand{\remarkname}{Remark}
\providecommand{\theoremname}{Theorem}

\begin{document}
\title{EDP-convergence for a linear reaction-diffusion system with fast reversible
reaction\thanks{Research supported by DFG via SFB 1114 (project no.235221301, subproject C05).}}
\author{Artur Stephan\thanks{Weierstraß-Institut für Angewandte Analysis und Stochastik, Mohrenstraße 39, 10117 Berlin, Germany, e-mail: \tt{artur.stephan@wias-berlin.de}}}
\date{\today}
\maketitle

\lhead{EDP-convergence for linear RDS}

\rhead{Artur Stephan}

\chead{\today}

\paragraph{Abstract:}

We perform a fast-reaction limit for a linear reaction-diffusion system
consisting of two diffusion equations coupled by a linear reaction.
We understand the linear reaction-diffusion system as a gradient flow
of the free energy in the space of probability measures equipped with
a geometric structure, which contains the Wasserstein metric for the
diffusion part and cosh-type functions for the reaction part. The
fast-reaction limit is done on the level of the gradient structure
by proving EDP-convergence with tilting. The limit gradient system
induces a diffusion system with Lagrange multipliers on the linear
slow-manifold. Moreover, the limit gradient system can be equivalently
described by a coarse-grained gradient system, which induces a diffusion
equation with a mixed diffusion constant for the coarse-grained slow
variable.

\section{Introduction}

Considering two species $X_{1}$ and $X_{2}$ which diffuse in a bounded
medium $\Omega\subset\R^{d}$ and react linearly $X_{1}\leftrightharpoons X_{2}$,
the evolution of their concentrations $c=(c_{1},c_{2})$ can be described
by the linear reaction-diffusion system\\
\begin{align}
\dot{c_{1}}=\delta_{1}\Delta c_{1}-(\tilde{\alpha}c_{1}-\tilde{\beta}c_{2})\nonumber \\
\dot{c_{2}}=\delta_{2}\Delta c_{2}+(\tilde{\alpha}c_{1}-\tilde{\beta}c_{2})\label{eq:LRDS-0}
\end{align}
complemented with no-flux boundary conditions and initial conditions,
where $\delta_{1},\delta_{2}>0$ are diffusion coefficients for species
$X_{1}$ and $X_{2}$, respectively, and $\tilde{\alpha},\tilde{\beta}>0$
are reaction rates describing the reaction speed of the linear reaction
$X_{1}\leftrightharpoons X_{2}$. The aim of the paper is to investigate
system \eqref{eq:LRDS-0} if the reaction is much faster than the
diffusion. To do this, we introduce a small parameter $\eps>0$ and
assume that the reaction rates are given by $\tilde{\alpha}=\tfrac{1}{\eps}\sqrt{\tfrac{\alpha}{\beta}}$,
$\tilde{\beta}=\tfrac{1}{\eps}\sqrt{\tfrac{\beta}{\alpha}}$. Then,
the system \eqref{eq:LRDS-0} can be rewritten in an $\eps$-dependent
reaction-diffusion system
\begin{align}
\dot{c_{1}^{\eps}} & =\delta_{1}\Delta c_{1}^{\eps}-\frac{1}{\eps}\left(\sqrt{\tfrac{\alpha}{\beta}}c_{1}^{\eps}-\sqrt{\tfrac{\beta}{\alpha}}c_{2}^{\eps}\right)\nonumber \\
\dot{c_{2}^{\eps}} & =\delta_{2}\Delta c_{2}^{\eps}+\frac{1}{\eps}\left(\sqrt{\tfrac{\alpha}{\beta}}c_{1}^{\eps}-\sqrt{\tfrac{\beta}{\alpha}}c_{2}^{\eps}\right)\ .\label{eq:LRDS}
\end{align}
Reaction systems and reaction-diffusion systems with slow and fast
time scales have attracted a lot of attention in the last years \cite{Evan80CTCDRS,HiHoPe00NDFR,Both03ILRCR,BotPie10QSSA,BotPie11ILRD,BoPiRo12CDLR,MurNin16FRL3CRDS,MieSte19ECLRS,DaDeJu20CDSFRL,MiPeSt20EDPCNLRS,PelRen20}.
Bothe and Hilhorst proved a fast-reaction limit $\eps\to0$ for \eqref{eq:LRDS}
in the following form.
\begin{thm}[\cite{BotHil02RDFRR}]
 Let $\Omega\subset\R^{d}$ be a domain with Lipschitz boundary.
Let $c_{1}^{\eps}$ and $c_{2}^{\eps}$ be weak solutions of \eqref{eq:LRDS}
with no-flux boundary conditions $\nabla c_{i}^{\eps}\cdot\nu=0$
on $\partial\Omega$. Then $c_{1}^{\eps}\rightarrow c_{1}$ and $c_{2}^{\eps}\rightarrow c_{2}$
in $\L^{2}([0,T]\times\Omega)$ as $\eps\rightarrow0$ and we have
$\frac{c_{1}}{\beta}=\frac{c_{2}}{\alpha}$. Moreover, defining the
coarse-grained concentration $\hat{c}=c_{1}+c_{2}$, then $\hat{c}$
solves the diffusion equation $\dot{\hat{c}}=\hat{\delta}\Delta\hat{c}$
with a new mixed diffusion coefficient $\hat{\delta}=\frac{\beta\delta_{1}+\alpha\delta_{2}}{\alpha+\beta}$.
\end{thm}

Essentially, the proof uses the free-energy as a Lyapunov function
to derive $\eps$-uniform bounds on the concentrations $c_{i}^{\eps}$
and their gradients $\nabla c_{i}^{\eps}$, which is then used to
prove convergence towards the slow manifold $\left\{ c\in[0,\infty[^{2}\ |\ \alpha c_{1}=\beta c_{2}\right\} $.
This proof also works for nonlinear reactions once $\eps$-uniform
$\L^{\infty}$-estimates are established (see \cite{BotHil02RDFRR}).
On the linear slow manifold, one easily verifies that the coarse grained
concentration $\hat{c}:=\frac{\alpha+\beta}{\beta}c_{1}=\frac{\alpha+\beta}{\alpha}c_{2}=c_{1}+c_{2}$
solves $\dot{\hat{c}}=\hat{\delta}\Delta\hat{c}$ where $\hat{\delta}=\frac{\beta\delta_{1}+\alpha\delta_{2}}{\alpha+\beta}$
is the effective mixed diffusion coefficient.

In this work, we are not primary interested in convergence of solutions
of system \eqref{eq:LRDS}. Instead, we perform the fast-reaction
limit on the level of the underlying variational structure, which
then implies convergence of solutions as a byproduct. Our starting
point is that reaction-diffusion systems such as \eqref{eq:LRDS}
can be written as a gradient flow equation induced by a gradient system
$(Q,\cE,\cR_{\eps}^{*})$, where the state space $Q$ is the space
of probability measures $Q=\Prob(\Omega\times\left\{ 1,2\right\} )$
and the driving functional is the free-energy ${\cal E}(\mu)=\int_{\Omega}\sum_{j=1}^{2}E_{B}\left(\frac{c_{j}}{w_{j}}\right)w_{j}\d x$
for measures $\mu=c\ \d x$, with the Boltzmann function $E_{B}(r)=r\log r-r+1$
and the (in general space dependent) stationary measure $w=(w_{1},w_{2})^{\mathrm{T}}$.
The dissipation potential $\cR_{\eps}^{*}$ that determines the geometry
of the underlying space is given by two parts $\cR_{\eps}^{*}=\cR_{\diff}^{*}+\cR_{\react,\eps}^{*}$
describing the diffusion and reaction separately. Since the pioneering
work of Otto and coauthors \cite{JoKiOt98VFFP,Otto01GDEE} it is known
that diffusion has to be understood as a gradient system driven by
the free-energy in the space of probability measures equipped with
the Wasserstein distance. The corresponding dissipation potential
$\cR_{\diff}^{*}$ is quadratic and given by 
\[
\cR_{\mathrm{diff}}^{*}(\mu,\xi)=\frac{1}{2}\int_{\Omega}\sum_{i=1}^{2}\delta_{i}\left|\nabla\xi_{i}\right|^{2}\ \d\mu_{i}.
\]
Later Mielke \cite{Miel11GSRD} proposed a quadratic gradient structure
also for reaction-diffusion systems with the same driving functional.
Geometric properties of that gradient structure were investigated
in \cite{LieMie13GSGC,GKMP20HODOT}. Here, we are not interested in
that gradient structure, but use a different, the so-called \textit{cosh-type
gradient structure}, where the reaction part is given by 
\[
\cR_{\mathrm{react},\eps}^{*}(\mu,\xi)=\frac{1}{\eps}\int_{\Omega}\mathsf{C}^{*}(\xi_{1}(x)-\xi_{2}(x))\ \sqrt{\d\mu_{1}\d\mu_{2}},
\]
 with $\mathsf{C}^{*}(r)=4(\mathrm{cosh}(r/2)-1)$. Setting $\cR_{\eps}^{*}=\cR_{\diff}^{*}+\cR_{\react,\eps}^{*}$,
the reaction-diffusion system \eqref{eq:LRDS} can now be written
as a gradient flow equation 
\[
\dot{\mu}=\partial_{\xi}\cR_{\eps}^{*}(\mu,-\D\cE(\mu)).
\]
Although there are many gradient structures for \eqref{eq:LRDS} (see
e.g. \cite[Sect. 4]{MieSte19ECLRS}) and the cosh-type gradient structure
entails several technical difficulties as defining a nonlinear kinetic
relation and not inducing metric on $Q$, it nevertheless has several
significant features. Historically, it has its origin in \cite{Marc15CECP}
where, following thermodynamic considerations, chemical reactions
are written in exponential terms. In recent years, the cosh-gradient
structure has been derived via a large-deviation principle \cite{MiPeRe14RGFL,MPPR17NETP},
and it was shown that it is stable under limit processes \cite{LMPR17MOGG}
that are similar to our approach. Moreover, it does not explicitly
depend on the stationary measure $w$, which allows for an rigorous
distinction between the energetic and dissipative part \cite{MieSte19ECLRS}.
This is physically reasonable because a change of the energy by an
external field should not influence the geometric structure of the
underlying space.

The goal of the paper is to construct an effective gradient system
$(Q,\cE,\cR_{\eff}^{*})$ and perform the limit $(Q,\cE,\cR_{\eps}^{*})\to(Q,\cE,\cR_{\eff}^{*})$
as $\eps\to0$. For this, we use the notion of \textit{convergence
of gradient systems in the sense of the energy-dissipation principle},
shortly called \textit{EDP-convergence}. EDP-convergence was introduced
in \cite{DoFrMi18GSWE} and further developed in \cite{MiMoPe20EFED,MieSte19ECLRS}
and is based on the dissipation functional
\[
\fD_{\eps}^{\eta}(\mu)=\int_{0}^{T}\cR_{\eps}(\mu,\dot{\mu})+\cR_{\eps}^{*}(\mu,\eta-\D\cE(\mu))\ \d t
\]
which, for solutions $\mu$ of the gradient flow equation describes
the total dissipation between initial time $\cE(\mu(0))$ and final
time $\cE(\mu(T))$, and can now be defined for general trajectories
$\mu\in\L^{1}([0,T],Q)$. The notion of EDP-convergence with tilting
requires $\Gamma$-convergences of the energies $\cE_{\eps}\Gammlim\cE_{0}$
and of the dissipation functionals $\fD_{\eps}^{\eta}\Gammlim\fD_{0}^{\eta}$
in suitable topologies, such that for all tilts $\eta$ the limit
$\fD_{0}^{\eta}$ has the form $\fD_{0}^{\eta}(\mu)=\int_{0}^{T}\cR_{\eff}(\mu,\dot{\mu})+\cR_{\eff}^{*}(\mu,\eta-\D\cE_{0}(\mu))\ \d t$,
see Section \ref{subsec:EDP-convergence} for a precise definition.
Importantly, the effective dissipation potential $\cR_{\eff}$ in
the $\Gamma$-limit is independent of the tilts, hence allowing for
extended energies. In our situation, the tilts $\eta$ correspond
to an external potential $V=(V_{1},V_{2})$ added to the energy $\cE$.
On the level of the PDE, the starting reaction-diffusion system is
extended to a reaction-drift-diffusion system of the form
\[
\frac{\d}{\d t}\begin{pmatrix}c_{1}\\
c_{2}
\end{pmatrix}=\div\left(\begin{pmatrix}\delta_{1}\nabla c_{1}\\
\delta_{2}\nabla c_{2}
\end{pmatrix}+\begin{pmatrix}\delta_{1}c_{1}\nabla V_{1}\\
\delta_{2}c_{2}\nabla V_{2}
\end{pmatrix}\right)+\frac{1}{\eps}\begin{pmatrix}-\sqrt{\frac{\alpha}{\beta}}\e^{\frac{V_{1}-V_{2}}{2}} & \sqrt{\frac{\beta}{\alpha}}\e^{\frac{V_{2}-V_{1}}{2}}\\
\sqrt{\frac{\alpha}{\beta}}\e^{\frac{V_{1}-V_{2}}{2}} & -\sqrt{\frac{\beta}{\alpha}}\e^{\frac{V_{2}-V_{1}}{2}}
\end{pmatrix}\begin{pmatrix}c_{1}\\
c_{2}
\end{pmatrix}.
\]

The main result of the paper is Theorem \ref{thm:EDP-convergence}
which asserts tilt EDP-convergence of $(Q,\cE,\cR_{\eps}^{*})$ to
$(Q,\cE,\cR_{\eff}^{*})$ as $\eps\to0$ where the effective dissipation
potential is given by
\begin{align*}
\cR_{\eff}^{*} & =\cR_{\diff}^{*}+\chi_{\left\{ \xi_{1}=\xi_{2}\right\} },
\end{align*}
where $\chi_{A}$ is the characteristic function of convex analysis
taking values zero in $A$ and infinity otherwise. The effective dissipation
potential describes diffusion but restricts the chemical potential
$\xi=(\xi_{1},\xi_{2})$ to a linear submanifold. The induced gradient
flow equation of the gradient system $(Q,\cE,\cR_{\eff}^{*})$ is
then given by a system of drift-diffusion equations on a linear submanifold
with a space and time dependent Lagrange multiplier $\lambda$
\begin{align*}
\begin{array}{cc}
\dot{c_{1}} & =\div\left(\delta_{1}\nabla c_{1}+\delta_{1}c_{1}\nabla V_{1}\right)-\lambda\\
\dot{c_{2}} & =\div\left(\delta_{2}\nabla c_{2}+\delta_{2}c_{2}\nabla V_{2}\right)+\lambda
\end{array}\quad,\quad\frac{c_{1}}{\beta\e^{-V_{1}}}=\frac{c_{2}}{\alpha\e^{-V_{2}}}\ .
\end{align*}
Moreover, as an immediate consequence of Theorem \ref{thm:EDP-convergence},
we obtain that the effective gradient flow equation can be equivalently
described as a drift-diffusion equation of the coarse-grained concentration
$\hat{c}$, see Proposition \ref{prop:DissipationFunctionalInCoarseGrainedVariables}.
Introducing the mixed diffusion coefficient $\hat{\delta}^{V}=\frac{\delta_{1}\beta\e^{-V_{1}}+\delta_{2}\alpha\e^{-V_{2}}}{\beta\e^{-V_{1}}+\alpha\e^{-V_{2}}}$
and the mixed potential $\hat{V}=-\log(\frac{\beta}{\alpha+\beta}\e^{-V_{1}}+\frac{\alpha}{\alpha+\beta}\e^{-V_{2}})$,
we obtain 
\[
\hat{c}=\div\left(\hat{\delta}^{V}\nabla\hat{c}+\hat{\delta}^{V}\hat{c}\nabla\hat{V}\right),
\]
which is in accordance with \cite{BotHil02RDFRR} in the potential-free
case $V=\mathrm{const}$. Moreover, we obtain a natural coarse-grained
gradient structure $(\hat{Q},\hat{\cE},\hat{\cR^{*}})$, where $\hat{Q}=\Prob(\Omega)$
is the coarse-grained state space and $\hat{\cE},\hat{\cR^{*}}$ are
the coarse-grained energy functional and dissipation potential, respectively.
Interestingly, this coarse-grained gradient structure $(\hat{Q},\hat{\cE},\hat{\cR^{*}})$
contains the same information as the effective gradient structure
$(Q,\cE,\cR_{\eff}^{*})$, although defined on a smaller state space,
see Proposition \ref{prop:DissipationFunctionalInCoarseGrainedVariables}.

The result on tilt EDP-convergence is an immediate consequence of
the $\Gamma$-convergence result of the dissipation functional $\fD_{\eps}^{\eta}$
(Theorem \ref{thm:LiminfEstimate}). The primal dissipation potential
$\cR_{\eps}$ is given by an infimal sum consisting of diffusion fluxes
and reaction fluxes coupled via a generalized continuity equation,
see Section \ref{subsec:DissipationFunctional}. Theorem \ref{thm:LiminfEstimate}
follows from the following observations: $\cR_{\eps}^{*}$ converges
monotonically to a singular limit $\cR_{\eff}^{*}$, the primal dissipation
potentials $\cR_{\eps}$ degenerate. It is not possible to control
the rates of $\dot{\mu}_{1}$ and $\dot{\mu}_{2}$ separately by $\cR_{\eps}$,
since the reaction flux between both species may become unbounded.
Instead, it is possible to prove compactness for the sum (or slow
variable) $\mu_{1}+\mu_{2}$ by $\cR_{\eps}$, and proving convergence
towards the slow-manifold where an equilibration takes place, i.e.
$\alpha c_{1}^{\eps}-\beta c_{2}^{\eps}\to0$. The two pieces of complementary
information provide strong convergence of the densities $c^{\eps}$
in $\L^{1}([0,T]\times\Omega)$. This procedure has been already successfully
applied for linear and nonlinear reaction systems \cite{MieSte19ECLRS,MiPeSt20EDPCNLRS}
and is here applied to a space-dependent evolution system. A posteriori
we conclude that the limit measure $\mu^{0}$ has indeed an absolutely
continuous representative using results from \cite{AmGiSa05GFMS}.
The construction of the recovery sequence relies on the fact that
the limit dissipation functional can be equivalently written as a
functional of coarse-grained variables. Only the reaction flux, which
is present for positive $\eps>0$ and hidden for $\eps=0$, has to
be reconstructed. One observes that diffusion causes the reaction
flux on an infinitesimally small scale. Since the dissipation functional
considers also fluctuations which may be not strictly positive and
not smooth in contrast to the solution of the linear reaction diffusion
system \eqref{eq:LRDS}, the construction of a recovery sequence is
completed by a suitable approximation argument.

Let us finally mention, that the same results can also be established
for reaction-diffusion systems, where more than two species are involved.
Applying the coarse-graining and reconstruction machinery as developed
in \cite{MieSte19ECLRS}, a similar $\Gamma$-convergence result for
the dissipation functional can be proved. For notational convenience
we restrict to the two-species situation and briefly discuss the multi-species
case in Section \ref{sec:Multispecies}. We refer also to \cite{Step21CGRCLRS},
where coarse-graining and reconstruction for concentrations as well
as the fluxes is developed.

\section{Gradient structures}

\subsection{Gradient systems and the energy-dissipation principle}

\label{SectionGeneralGS}

Let us briefly recall what we mean with a gradient system. Following
\cite{Miel16EGCG}, we call a triple $(Q,\cE,\cR)$ a \textsl{gradient
system} if 
\begin{enumerate}
\item[(1)]  $Q$ is a closed convex subset of a Banach space $X$
\item[(2)] $\cE:Q\rightarrow\R_{\infty}:=\R\cup\{\infty\}$ is a functional
(such as the free energy)
\item[(3)] $\cR:Q\times X\rightarrow\R_{\infty}$ is a dissipation potential,
which means that for any $u\in Q$ the functional $\cR(u,\cdot):X\rightarrow\R_{\infty}$
is lower semicontinuous (lsc), nonnegative and convex, and it satisfies
$\cR(u,0)=0$. 
\end{enumerate}
We define the dual dissipation potential $\cR^{*}:Q\times X^{*}\to[0,\infty]$
using the Legendre transform via 
\[
\cR^{*}(u,\xi)=(\cR(u,\cdot))^{*}(\xi)=\sup_{v\in X}\left\{ \langle v,\xi\rangle-\cR(u,v)\right\} \ .
\]
The gradient system is uniquely described by $(Q,\cE,\cR)$ or, equivalently
by $(Q,\cE,\cR^{*})$ and, in particular, in this paper we use the
second representation.

The dynamics of a gradient system can be formulated in different ways
as an equation in $X$, in $\R$ or in $X^{*}$ (the dual Banach space
of $X$), respectively: 
\begin{enumerate}
\item[(1)] \textbf{Force balance in $X^{*}$:} $0\in\partial_{\dot{u}}\cR(u,\dot{u})+\D\cE(u)\in X^{*}$, 
\item[(2)] \textbf{ Power balance in $\R$:} $\cR(u,\dot{u})+\cR^{*}(u,-\D\cE(u))=-\la\D\cE(u),\dot{u}\ra$, 
\item[(3)] \textbf{ Rate equation in $X$:} $\dot{u}\in\partial_{\xi}\cR^{*}(u,-\D\cE(u))\in X$.
\end{enumerate}
(Here, $\partial$ denotes the subdifferential of convex analysis.)
Equations (1) and (3) are called \textit{gradient flow equation} associated
with $(Q,\cE,\cR^{*})$. The equivalent formulations rely on the following
fact: Let $X$ be a reflexive Banach space and $\Psi:X\rightarrow\R_{\infty}$
be a proper, convex and lsc. Then for every $\xi\in X^{*}$ and $v\in X$
the following five statements, the so-called \textit{Legendre-Fenchel-equivalences},
are equivalent:
\begin{center}
\begin{tabular}{lcr}
$v\in\mathrm{Argmin}_{w\in X}(\Psi(w)-\la\xi,w\ra)$ & $\Leftrightarrow$ & $\Psi(v)+\Psi^{*}(\xi)=\la\xi,v\ra$\tabularnewline
$~~~~~~~~~~\Leftrightarrow$ & $\Psi(v)+\Psi^{*}(\xi)=\la\xi,v\ra$ & $\Leftrightarrow$~~~~~\tabularnewline
$v\in\partial\Psi^{*}(\xi)$ & $\Leftrightarrow$ & $\xi\in\mathrm{Argmin}_{\eta\in X^{*}}(\Psi^{*}(\eta)-\la\eta,v\ra)$\tabularnewline
\end{tabular}
\par\end{center}

\begin{center}
Especially the second dynamic formulation, the power balance (2),
is interesting for us. Integrating the power balance (2) in time form
0 to $T$ and using the chain rule for the time-derivative of $t\mapsto\cE(u(t))$,
we get another equivalent formulation of the dynamics of the gradient
system, which is called \textit{Energy-Dissipation-Balance}: 
\begin{align}
\mathrm{(EDB)}~~~~~~~\cE(u(T))+\int_{0}^{T}\left[\cR(u,\dot{u})+\cR^{*}(u,-\D\cE(u))\right]\d t=\cE(u(0)).\label{eq:EDB}
\end{align}
Equation (EDB) compares the energy of the system at time $t=0$ and
at time $t=T$, the difference is described by the total dissipation
from $t=0$ to $t=T$. This gives rise to another definition: We define
the\textit{ De Giorgi dissipation functional} as 
\begin{align*}
\fD(u)=\int_{0}^{T}\left[\cR(u,\dot{u})+\cR^{*}(u,-\D\cE(u))\right]\d t,
\end{align*}
for $u\in\W^{1,1}([0,T],Q)$ and extend it to infinity otherwise.
The following \textit{energy-dissipation principle} provides the definition
for solutions of the gradient flow equation, see e.g. \cite[Prop.  1.4.1]{AmGiSa05GFMS},
\cite[Def. 1.1]{AMPSV12PLWG}, \cite[Thm 2.5]{MiMoPe20EFED}.
\par\end{center}
\begin{defn}
\begin{flushleft}
\label{def:EDP}We say $u\in\W^{1,1}([0,T],Q)$ is a solution of the
gradient flow equation (1) or (3) induced by the gradient system $(Q,\cE,\cR^{*})$,
if $\cE(u(0))<\infty$ and the energy-dissipation balance holds.
\par\end{flushleft}
\end{defn}

\subsection{Definition of EDP-convergence \label{subsec:EDP-convergence}}

The definition of EDP-convergence for gradient systems relies on the
notion of $\Gamma$-conver-\\
gence for functionals (cf. \cite{Dalm93IGC}). If $Y$ is a Banach
space and $I_{\eps}:Y\to\R_{\infty}$ we write $I_{\eps}\Gammlim I_{0}$
and $I_{\eps}\wGammlim I_{0}$ for $\Gamma$-convergence in the strong
and weak topology, respectively. If both holds this is called Mosco-convergence
and written as $I_{\eps}\Moscolim I_{0}$. 

For families of gradient systems $(X,\cE_{\eps},\cR_{\eps})$, three
different levels of EDP-convergence are introduced and discussed in
\cite{DoFrMi18GSWE,MiMoPe20EFED}, called simple EDP-convergence,
EDP-convergence with tilting and contact EDP-convergence with tilting
EDP-convergence with tilting is the strongest notion, since it implies
the other two notions. Here we will only use the first two notions.
For all three notions the choice of weak or strong topology is still
to be decided according to the specific problem.
\begin{defn}[Simple EDP-convergence]
\label{def:EDPcvg} A family of gradient structures $(Q,\cE_{\eps},\cR_{\eps})$
is said to \textit{EDP-converge} to the gradient system $(Q,\cE_{\mathrm{0}},\cR_{\mathrm{eff}})$
if the following conditions hold:
\begin{enumerate}
\item $\cE_{\eps}\Gammlim\cE_{0}$ on $Q\subset X$;
\item $\fD_{\eps}$ strongly $\Gamma$-converges to $\fD_{0}$ on $\L^{1}([0,T],Q)$
conditioned to bounded energies (we write $\fD_{\eps}\GamLimE\fD_{0}$),
i.e. we
\begin{enumerate}
\item (Liminf-estimate) For all strongly converging families $u_{\eps}\to u$
in $\L^{1}([0,T],Q)$ which satisfy $\sup_{\eps>0}\mathrm{ess\,sup}_{t\in[0,T]}\cE_{\eps}(u_{\eps}(t))<\infty$,
we have $\liminf_{\eps\to0^{+}}\fD_{\eps}(u_{\eps})\geq\fD_{0}(u)$,
\item (Limsup-estimate) For all $\widetilde{u}\in\L^{1}([0,T],Q)$ there
exists a strongly converging family $u_{\eps}\to\widetilde{u}$ in
$\L^{1}([0,T],Q)$ with $\sup_{\eps>0}\mathrm{ess\,sup}_{t\in[0,T]}\cE_{\eps}(\widetilde{u}_{\eps}(t))<\infty$
such that we have $\limsup_{\eps\to0^{+}}\fD_{\eps}(u_{\eps})\leq\fD_{0}(\widetilde{u})$.
\end{enumerate}
\item There is an effective dissipation potential $\cR_{\mathrm{eff}}:Q\times X\to\R_{\infty}$
such that $\fD_{0}$ takes the form of a dual sum, namely $\fD_{0}(u)=\int_{0}^{T}\{\cR_{\mathrm{eff}}(u,\dot{u}){+}\cR_{\mathrm{eff}}^{*}(u,-\D\cE_{\mathrm{eff}}(u))\}\d t$.
\end{enumerate}
\end{defn}

Similarly, one can also use weak $\Gamma$- or Mosco-convergence conditioned
to bounded energy, which we will then write as $\fD_{\eps}\wGamLimE\fD_{0}$
and $\fD_{\eps}\MoscoLimE\fD_{0}$. In fact, for our fast-slow reaction
systems we are going to prove $\fD_{\eps}\MoscoLimE\fD_{0}$.

A general feature of EDP-convergence is that, under suitable conditions,
solutions $u$ of the gradient flow equation $\dot{u}=\partial_{\xi}\cR_{\eff}^{*}(u,{-}\D\cE_{0}(u))$
of the effective gradient system $(X,\cE_{0},\cR_{\eff})$ are indeed
limits of solutions $u^{\eps}$ of the gradient flow equation $\dot{u}=\partial_{\xi}\cR_{\eps}^{*}(u,{-}\D\cE_{\eps}(u))$,
see e.g. \cite[Thm. 11.3]{Brai13LMVE}, \cite[Lem. 3.4]{MieSte19ECLRS}
and \cite[Lem. 2.8]{MiMoPe20EFED}.

A strengthening of simple EDP-convergence is the so-called \textit{EDP-convergence
with tilting}. This notion involves the tilted energy functionals
$\cE_{\eps}^{\eta}:Q\ni u\mapsto\cE_{\eps}(u)-\langle\eta,u\rangle$,
where the tilt $\eta$ (also called \textit{external loading}) varies
through the whole dual space $X^{*}$. 
\begin{defn}[{EDP-convergence with tilting (cf. \cite[Def. 2. 14]{MiMoPe20EFED}}]
A family of gradient structures $(Q,\cE_{\eps},\cR_{\eps})$ is said
to \textit{EDP-converge with tilting} to the gradient system $(Q,\cE_{0},\cR_{\eff})$,
if for all tilts $\eta\in X^{*}$ we have $(Q,\cE_{\eps}^{\eta},\cR_{\eps})$
EDP-converges to $(Q,\cE_{\eps}^{\eta},\cR_{\eff})$.
\end{defn}

Clearly, we have that $\cE_{\eps}\Gammlim\cE_{0}$ implies $\cE_{\eps}^{\eta}\Gammlim\cE_{0}^{\eta}$
for all $\eta\in X^{*}$ (and similarly for weak $\Gamma$-convergence),
since the linear tilt $u\mapsto-\langle\eta,u\rangle$ is weakly continuous.
The main and nontrivial assumption is that additionally 
\[
\fD_{\eps}^{\eta}:u\mapsto\int_{0}^{T}\!\!\big\{\cR_{\eps}(u,\dot{u})+\cR_{\eps}^{*}(u,\eta{-}\D\cE_{\eps}(u))\big\}\d t
\]
$\Gamma$-converges in $\L^{1}([0,T],Q)$ to $\fD_{0}^{\eta}$ for
all $\eta\in X^{*}$ and that this limit $\fD_{0}^{\eta}$ is given
in $\cR\oplus\cR^{*}$-form with $\cR_{\eff}$ via 
\[
\fD_{0}^{\eta}(u)=\int_{0}^{T}\!\!\big\{\cR_{\mathrm{eff}}(u,\dot{u})+\cR_{\mathrm{eff}}^{*}(u,\eta{-}\D\cE_{\mathrm{eff}}(u))\big\}\d t.
\]
The main point is that $\cR_{\eff}$ remains independent of $\eta\in X^{*}$.
We refer to \cite{MiMoPe20EFED} for a discussion of this and the
other two notions of EDP-convergence.

\section{Gradient system of reaction-diffusion systems\label{sec:GSforLRDS}}

In this section, we present the gradient system $(Q,\cE,\cR_{\eps}^{*})$,
which induces the reaction-diffusion system \eqref{eq:LRDS}. In Section
\ref{subsec:TiltedGFE} we derive the gradient flow equation of the
gradient system including general tilts of the energy. In Section
\ref{subsec:DissipationFunctional} we compute the primal dissipation
potential $\cR_{\eps}$, which is only implicitly given via a infimal-convolution,
and the total dissipation functional $\fD_{\eps}^{\eta}$, which will
be the main object of interest in Section \ref{sec:EDP-convergence-result}.
In Section \ref{sec:GSforLRDS}, the computations are basically formal;
the precise functional analytic setting is presented in Section \ref{sec:EDP-convergence-result}
which also includes the $\Gamma$-convergence and EDP-convergence
result.

\subsection{Gradient structure for the linear reaction system}

Although a gradient system induces a unique gradient flow equation,
a general evolution equation can often be described by many different
gradient systems. The choice of the gradient structure is a question
of modeling since it adds thermodynamic information to the system,
which is not inherent in the evolution equation itself. Here, we follow
the pioneering work of Otto and coauthors \cite{JoKiOt98VFFP,Otto01GDEE}
who showed that certain diffusion type equations can be understood
as a gradient flow equation of the free energy in the space of probability
measures equipped with the Wasserstein distance. Later Mielke proposed
a gradient structure for a reaction diffusion system satisfying detailed
balance \cite{Miel11GSRD}. For a system with two species with a reversible
reaction detailed balance is always satisfied. For the reaction part,
we use the gradient structure which has been derived via a large-deviation
principle from a microscopic Markov process in \cite{MiPeRe14RGFL}.
We refer also to \cite{Reng18GGSF}, where our choice of gradient
structure has been formally derived.

The gradient system $(Q,\cE,\cR_{\eps}^{*})$ is defined as follows:
The state space is the space of probability measure on $Q\times\left\{ 1,2\right\} $
\begin{align*}
Q & :=\mathrm{Prob}(\Omega\times\left\{ 1,2\right\} )=\{\mu=(\mu_{1},\mu_{2})\in\R^{2}:\mu_{i}\in{\cal M}(\Omega),\ \mu_{i}\geq0,~\sum_{i=1}^{2}\mu_{i}(\Omega)=1\},
\end{align*}
where we assume that $\Omega\subset\R^{d}$ is a compact domain with
normalized mass $|\Omega|=1$. The driving energy functional $\cE:Q\rightarrow\R_{\infty}:=\R\cup\{\infty\}$
is the free-energy of the reaction-diffusion system. It is finite
for measures $\mu=(\mu_{1},\mu_{2})$ with Lebesgue density $c=(c_{1},c_{2})$
only and has the form 
\begin{align}
\cE(\mu):=\begin{cases}
\int_{\Omega}\sum_{j=1}^{2}E_{B}\left(\frac{c_{j}}{w_{j}}\right)w_{j}\d x, & \mathrm{~if~}\mu=c\cdot\d x\\
\infty, & \mathrm{otherwise}.
\end{cases}\label{eq:Energy}
\end{align}
where the Boltzmann function is defined as $E_{B}(r)=r\log r-r+1$
and the positive stationary measure is given by $w=\frac{1}{\alpha+\beta}(\beta,\alpha)^{\mathrm{T}}$.
Note that the stationary measure $w$ as well as the energy ${\cal E}$
is $\eps$-independent. The derivative of the energy ${\cal E}$ is
only defined in its domain, i.e. for measures with Lebesgue density
$c$, and has the form
\[
\D\cE(\mu)=\sum_{j=1}^{2}\left(\log c_{j}-\log w_{j}\right)=\sum_{j=1}^{2}\left(\log\frac{c_{j}}{w_{j}}\right)\ .
\]

As the equation splits into a diffusion and reaction part, so does
the dual dissipation functional. We define
\begin{align*}
\cR_{\eps}^{*}(\mu,\xi):=\cR_{\mathrm{diff}}^{*}(\mu,\xi)+\cR_{\mathrm{react},\eps}^{*}(\mu,\xi)
\end{align*}
where 
\begin{align*}
\cR_{\mathrm{diff}}^{*}(\mu,\xi) & :=\frac{1}{2}\int_{\Omega}\sum_{j=1}^{2}\delta_{j}|\nabla\xi_{j}(x)|^{2}\d\mu_{j},\\
\cR_{\mathrm{react},\eps}^{*}(\mu,\xi) & :=\frac{1}{\eps}\int_{\Omega}\mathsf{C}^{*}(\xi_{1}(x)-\xi_{2}(x))\ \d\sqrt{\mu_{1}\mu_{2}},
\end{align*}
where we use the cosh-function $\mathsf{C}^{*}(x)=4\left(\cosh(x/2)-1\right)$
and for measures $\mu$ with Lebesgue density $c$ we have $\d\sqrt{\mu_{1}\mu_{2}}:=\sqrt{c_{1}c_{2}}\d x$.

The diffusion part ${\cal R}_{\diff}^{*}$ induces the Wasserstein
distance on $Q$. The $\eps$-dependent reaction part ${\cal R}_{\mathrm{react},\eps}^{*}$
forces the evolution close to a linear submanifold given by 
\[
\cR_{\react,\eps}^{*}(\mu,-\D\cE(\mu))=0\Leftrightarrow\alpha c_{1}-\beta c_{2}=0\,.
\]
Note, that since ${\cal R}_{\mathrm{react},\eps}^{*}$ is not 2-homogeneous,
it does not define a metric on $Q$. We refer to \cite{PRST20JPGGF}
which treats similar and general dissipation potentials and understands
them as generalized transport costs on discrete spaces. Note that
${\cal R}_{\eps}^{*}$ does not depend on the stationary measure $w$
explicitly, as highlighted in \cite{MieSte19ECLRS}.

\subsection{The tilted gradient flow equation\label{subsec:TiltedGFE}}

In this section, we derive the gradient flow equation of the gradient
system $(Q,\cE,\cR_{\eps}^{*})$. To exploit the full information
of the dissipation potential, we consider general tilted energies.
First, we present how a change of energy by a linear tilt corresponds
to a change of stationary measure, and secondly, we compute the induced
gradient flow equation.

Let us first consider two free energies \eqref{eq:Energy} with different
stationary measures $w,\widetilde{w}$, which may be space dependent
but are assumed to be positive. Assuming a density $\mu=c\,\d x$
and using $\sum_{i}\int_{\Omega}w_{i}\d x=\sum_{i}\int_{\Omega}c_{i}\d x=1$
(where we used $|\Omega|=1$), we have 
\[
\cE(\mu)=\sum_{i=1}^{2}\int_{\Omega}E_{B}\left(\frac{c_{i}}{w_{i}}\right)w_{i}\d x=\sum_{i=1}^{2}\int_{\Omega}\left\{ c_{i}\log c_{i}-c_{i}\log w_{i}\right\} \d x.
\]
In particular, we conclude that $\widetilde{{\cal E}}(\mu)+\sum_{i=1}^{2}\int_{\Omega}c_{i}\log\widetilde{w}_{i}\d x={\cal E}(\mu)+\sum_{i=1}^{2}\int_{\Omega}c_{i}\log w_{i}\d x$
which implies
\[
\widetilde{{\cal E}}(\mu)={\cal E}(\mu)+\sum_{i=1}^{2}\int_{\Omega}c_{i}\log\left(\frac{w_{i}}{\widetilde{w}_{i}}\right)\d x.
\]
Hence, changing the underlying stationary measure corresponds to a
linear tilt of the energy by a two component potential $V=(V_{1},V_{2})$
where $V_{i}=\log\left(\frac{w_{i}}{\widetilde{w}_{i}}\right)$. On
the other hand, a tilted energy has a different stationary measure
as its minimum. To compute the new stationary measure, we introduce
tilted energies 
\[
{\cal E}^{V}(\mu):={\cal E}(\mu)+\sum_{i=1}^{2}\int_{\Omega}V_{i}\d\mu_{i},
\]
where $V\in\mathrm{C}^{1}(\Omega,\R^{2})$ is a two component smooth
potential. Moreover, we introduce $\eta_{i}:=\e^{-V_{i}}$ and clearly,
we have $\eta_{i}>0$ on $\Omega\subset\R^{d}$. We compute the stationary
state $w^{V}$ by minimizing ${\cal E}^{V}$ on the space $Q=\mathrm{Prob}(\Omega\times\left\{ 1,2\right\} )$.
We obtain the space dependent stationary measure
\begin{align}
w_{i}^{V}=\frac{1}{Z}w_{i}\e^{-V_{i}},\ \mathrm{where}\ Z:=\sum_{i=1}^{2}\int_{\Omega}w_{i}\e^{-V_{i}}\d x\ .\label{eq:StationaryMeasure}
\end{align}

Next, we compute the tilted gradient flow equation $\dot{c}=\partial_{\xi}{\cal R}_{\eps}^{*}(\mu,-\D\cE^{V}(\mu))$,
which is induced by the gradient system $(Q,\cE^{V},\cR_{\eps}^{*}=\cR_{\diff}^{*}+\cR_{\react,\eps}^{*})$.
First, we observe that ${\cal E}(\mu)<\infty$ if and only if ${\cal E}^{V}(\mu)<\infty$.
Inserting $\xi_{i}=\left(-\D\cE^{V}(\mu)\right)_{i}$ into $\partial_{\xi}{\cal R}_{\diff,\eps}^{*}(\mu,\xi)$,
we see that
\[
\partial_{\xi}\cR_{\mathrm{diff}}^{*}(\mu,\cdot)|_{\xi=-\D\cE^{V}(\mu)}=-(\div(\delta_{i}c_{i}\nabla(-\log(c_{i}/w_{i})-V_{i}))_{i=1,2}=\div\left(\delta_{i}\nabla c_{i}+\delta_{i}c_{i}\nabla V_{i}\right)_{i=1,2},
\]
which is a system of two uncoupled drift-diffusion equations or Fokker-Planck
equations for the concentrations $c_{i}$ where the fluxes are given
by a diffusion part $-\delta_{i}\nabla c_{i}$ and a drift part $-\delta_{i}c_{i}\nabla V_{i}$.

For the reaction part of the dual dissipation potential, we insert
$\xi_{i}=\left(-\D\cE^{V}(\mu)\right)_{i}$ into $\partial_{\xi}{\cal R}_{\mathrm{react},\eps}^{*}(\mu,-\D\cE^{V}(\mu))$.
On readily verifies the identity $\left(\mathsf{C}^{*}\right)'(\log p-\log q)=\frac{p-q}{\sqrt{pq}}$
for the cosh-function and conclude
\begin{align*}
\sqrt{c_{1}c_{2}}\left(\mathsf{C}^{*}\right)'(\xi_{1}(x)-\xi_{2}(x))|_{\xi=-\D\cE^{V}(\mu)} & =\sqrt{c_{1}c_{1}}\frac{\frac{c_{2}}{w_{2}\eta_{2}}-\frac{c_{1}}{w_{1}\eta_{1}}}{\sqrt{\frac{c_{1}}{w_{1}\eta_{1}}\frac{c_{2}}{w_{2}\eta_{2}}}}=\sqrt{w_{1}\eta_{1}w_{2}\eta_{2}}\left(\frac{c_{2}}{w_{2}\eta_{2}}-\frac{c_{1}}{w_{1}\eta_{1}}\right).
\end{align*}
Hence, we get
\begin{align*}
\partial_{\xi_{1}}\cR_{\mathrm{react},\eps}^{*}(\mu,\cdot)|_{\xi=-\D\cE^{V}(\mu)} & =-\partial_{\xi_{2}}\cR_{\mathrm{react}}^{*}(\mu,\cdot)|_{\xi=-\D\cE^{V}(\mu)}=\frac{1}{\eps}\left(\sqrt{\frac{\beta}{\alpha}}\sqrt{\frac{\eta_{1}}{\eta_{2}}}c_{2}-\sqrt{\frac{\alpha}{\beta}}\sqrt{\frac{\eta_{2}}{\eta_{1}}}c_{1}\right),
\end{align*}
which is linear in $c=(c_{1},c_{2})$. In vector notation, we get
a tilted Markov generator of the form
\[
\partial_{\xi}{\cal R}_{\mathrm{react},\eps}^{*}(\mu,-\D\cE^{V}(\mu))=\frac{1}{\eps}\begin{pmatrix}-\sqrt{\frac{\alpha}{\beta}\frac{\eta_{2}}{\eta_{1}}} & \sqrt{\frac{\beta}{\alpha}\frac{\eta_{1}}{\eta_{2}}}\\
\sqrt{\frac{\alpha}{\beta}\frac{\eta_{2}}{\eta_{1}}} & -\sqrt{\frac{\beta}{\alpha}\frac{\eta_{1}}{\eta_{2}}}
\end{pmatrix}c=\frac{1}{\eps}\begin{pmatrix}-\sqrt{\frac{\alpha}{\beta}}\e^{\frac{V_{1}-V_{2}}{2}} & \sqrt{\frac{\beta}{\alpha}}\e^{\frac{V_{2}-V_{1}}{2}}\\
\sqrt{\frac{\alpha}{\beta}}\e^{\frac{V_{1}-V_{2}}{2}} & -\sqrt{\frac{\beta}{\alpha}}\e^{\frac{V_{2}-V_{1}}{2}}
\end{pmatrix}c
\]
 which has the space dependent stationary measure 
\[
w^{V}=\frac{1}{Z(\alpha+\beta)}(\beta\eta^{1},\alpha\eta^{2})^{\mathrm{T}}=\frac{1}{Z}(w_{1}\e^{-V_{1}},w_{2}\e^{-V_{2}})^{\mathrm{T}}\,.
\]
Summarizing, the tilted evolution equation has the form
\begin{equation}
\frac{\d}{\d t}\begin{pmatrix}c_{1}\\
c_{2}
\end{pmatrix}=\div\left(\begin{pmatrix}\delta_{1}\nabla c_{1}\\
\delta_{2}\nabla c_{2}
\end{pmatrix}+\begin{pmatrix}\delta_{1}c_{1}\nabla V_{1}\\
\delta_{2}c_{2}\nabla V_{2}
\end{pmatrix}\right)+\frac{1}{\eps}\begin{pmatrix}-\sqrt{\frac{\alpha}{\beta}}\e^{\frac{V_{1}-V_{2}}{2}} & \sqrt{\frac{\beta}{\alpha}}\e^{\frac{V_{2}-V_{1}}{2}}\\
\sqrt{\frac{\alpha}{\beta}}\e^{\frac{V_{1}-V_{2}}{2}} & -\sqrt{\frac{\beta}{\alpha}}\e^{\frac{V_{2}-V_{1}}{2}}
\end{pmatrix}\begin{pmatrix}c_{1}\\
c_{2}
\end{pmatrix},\label{eq:RDDS}
\end{equation}
which is a linear reaction-drift-diffusion system with space dependent
reaction rates. In the special case without external forcing $V=\mathrm{const}$,
we get the linear reaction diffusion system \eqref{eq:LRDS}. Note
that the reaction part still inherits symmetry since the product of
the off-diagonal elements is constant in space. In particular, not
all general linear reaction-drift-diffusion system with space dependent
reaction rates for two species can be expressed in the form \eqref{eq:RDDS}
and are induced by the gradient system $(Q,\cE^{V},\cR_{\eps}^{*}=\cR_{\diff}^{*}+\cR_{\react,\eps}^{*})$.

\subsection{The dissipation functional\label{subsec:DissipationFunctional}}

In this section, we compute the dissipation functional $\fD_{\eps}$,
which consists of two parts: the velocity part given by the primal
dissipation potential $\cR_{\eps}$ and the slope-part (sometimes
also called \textit{Fisher information}) $\cR_{\eps}^{*}(\mu,-\D\cE(\mu))$.
Here, all computations are formal and we always assume that the measure
$\mu$ has a Lebesgue density $c$. The precise functional analytic
setting is presented in the Section \ref{sec:EDP-convergence-result}.

The primal dissipation potential $\cR_{\eps}$, given by the Legendre
transform of the dual dissipation potential $\cR_{\eps}^{*}=\cR_{\diff}^{*}+\cR_{\eps,\react}^{*}$,
can be computed via inf-convolution of $\cR_{\diff}$ and $\cR_{\react,\eps}$.
First, we compute both primal dissipation potentials separately. To
do this, we introduce the following notation: For a convex, lsc. function
$\mathsf{F}:X\rightarrow[0,\infty]$ on a reflexive and separable
Banach space $X$ with Legendre dual $\mathsf{F}^{*}$, we define
the function $\widetilde{\mathsf{F}}:[0,\infty[\times X\to[0,\infty]$
by 
\[
\widetilde{\mathsf{F}}(a,x):=\left(a\,\mathsf{F}^{*}(\cdot)\right)^{*}(x)=\begin{cases}
a\,\mathsf{F}\left(\frac{1}{a}\,x\right) & \mathrm{for\ }a>0\ ,\\
\chi_{0}(x) & \mathrm{for\ }a=0\ .
\end{cases}
\]
Introducing the quadratic function $\mathsf{Q}(x)=\frac{1}{2}|x|^{2}$
on $\R^{d}$, the primal dissipation potential of the diffusion part
${\cal R_{\diff}^{*}}$ is given by

\[
\cR_{\mathrm{diff}}(\mu,v)=\sum_{j=1}^{2}\int_{\Omega}\widetilde{\mathsf{Q}}\left(\delta_{j}c_{j},J_{j}\right)\d x,
\]
where $J_{j}$ is, by definition, the unique solution of the elliptic
equation $v_{j}+\div J_{j}=0$ with $J\cdot\nu=0$ on $\partial\Omega$.
For positive $c_{j}$, we have $\widetilde{\mathsf{Q}}\left(\delta_{j}c_{j},J_{j}\right)=\frac{1}{2}\frac{|J_{j}|^{2}}{\delta_{j}c_{j}}$.

The primal dissipation potential of the reaction part is

\[
\cR_{\mathrm{react,\eps}}(\mu,b)=\begin{cases}
\int_{\Omega}\widetilde{\mathsf{C}}\left(\frac{\sqrt{c_{1}c_{2}}}{\eps},b_{2}\right)\d x, & \mathrm{for\ }b_{1}+b_{2}=0\\
\infty & \mathrm{for\ }b_{1}+b_{2}\neq0
\end{cases}\ ,
\]
where $\mathsf{C}=\left(\mathsf{C}^{*}\right)^{*}$ is the Legendre
transform of the cosh-function $\mathsf{C}^{*}(x)=4\left(\cosh(x/2)-1\right)$.
In the following, we use the inequality\textcolor{black}{
\begin{equation}
\frac{1}{2}|r|\cdot\log(|r|+1)\leq\mathsf{C}(r)\leq2|r|\cdot\log(|r|+1),\label{eq:InequalityForC}
\end{equation}
 which, in particular, implies that the Orlicz class for $A\subset\R^{d}$
given by 
\[
\widetilde{\L}^{\mathsf{C}}(A):=\{u\in\L^{1}(A):\int_{A}\mathsf{C}(u)\d x<\infty\},
\]
is, in fact, a Banach space} $\widetilde{\L}^{\mathsf{C}}(A)=\L^{\mathsf{C}}(A)$
with the norm $\|u\|_{\mathsf{C}}=\sup_{\int_{A}\mathsf{C}^{*}(v)\leq1}\left|\int_{A}uv\ \d x\right|$.

Importantly, functions $\widetilde{\mathsf{Q}}$, $\widetilde{\mathsf{C}}$
as well as the functionals $\cR_{\diff},\cR_{\react,\eps}$ are convex
on their domain of definition.

The primal dissipation potential $\cR_{\eps}$ is the inf-convolution
of ${\cal R}_{\mathrm{diff}}$ and ${\cal R}_{\mathrm{react},\eps}$,
and is given by 
\begin{align*}
\cR_{\eps} & (\mu,v)=\inf_{v=u_{1}+u_{2}}\left\{ \cR_{\mathrm{diff}}(\mu,u_{1})+\cR_{\mathrm{react},\eps}(\mu,u_{2})\right\} \\
=\inf_{J,b} & \left\{ \sum_{j=1}^{2}\int_{\Omega}\widetilde{\mathsf{Q}}\left(\delta_{j}c_{j},J_{j}\right)\d x+\int_{\Omega}\widetilde{\mathsf{C}}\left(\tfrac{\sqrt{c_{1}c_{2}}}{\eps},b_{2}(x)\right)\d x:\left\{ \begin{array}{c}
v_{1}=-\div J_{1}+b_{1}\\
v_{2}=-\div J_{2}+b_{2}\\
b_{1}+b_{2}=0
\end{array}\right\} \right\} .
\end{align*}
In time-integrated form we get for $v=\dot{\mu}$ that 
\begin{align*}
\int_{0}^{T}\cR_{\eps} & (\mu,\dot{\mu})\ \d t=\inf_{v=v_{1}+v_{2}}\int_{0}^{T}\left\{ \cR_{\mathrm{diff}}(\mu,v_{1})+\cR_{\mathrm{react},\eps}(\mu,v_{2})\right\} \d t\\
=\inf_{J,b} & \left\{ \int_{0}^{T}\left\{ \sum_{j=1}^{2}\int_{\Omega}\widetilde{\mathsf{Q}}\left(\delta_{j}c_{j},J_{j}\right)\d x+\int_{\Omega}\widetilde{\mathsf{C}}\left(\tfrac{\sqrt{c_{1}c_{2}}}{\eps},b_{2}(x)\right)\right\} \d x\ \d t:\ (c,J,b)\in\mathrm{(gCE)}\right\} .
\end{align*}
where we introduce the notation of a (linear) generalized continuity
equation
\[
(c,J,b)\in\mathrm{(gCE)}\ \ \Leftrightarrow\ \ \left\{ b_{1}+b_{2}=0\ \mathrm{and}\ \left\{ \begin{array}{c}
\dot{c}_{1}=-\div J_{1}+b_{1}\\
\dot{c}_{2}=-\div J_{2}+b_{2}
\end{array}\right\} \right\} .
\]
Without the reaction part, $\int_{0}^{T}{\cal R}_{\eps}\d t$ is the
dynamic formulation à la Benamou-Brenier of the Wasserstein distance
in $Q$ \cite{BenBre00CFMS}, which can be equivalently written in
the form
\[
{\cal W}_{2}(\mu_{0},\mu_{1})^{2}=\inf\left\{ \int_{0}^{1}\int_{\Omega}\sum_{j=1}^{2}\delta_{j}|v_{j}|^{2}\d\mu_{j}:\dot{\mu}_{j}+\div(\mu_{j}v_{j})=0,\ \mu_{j,0}=\mu_{0},\ \mu_{j,1}=\mu_{1}\right\} 
\]
expressed in terms of transport velocities $v_{j}=J_{j}/c_{j}$. The
Wasserstein distance can be interpreted as a cost in transporting
mass from one measure $\mu_{0}$ to $\mu_{1}$. In our situation $\int_{0}^{T}{\cal R}_{\eps}\d t$
is jointly convex in $c$, $J$ and $b$ and corresponds to modified
cost function which also takes the reaction fluxes into account. The
optimal diffusion fluxes $J_{j}$ and reaction fluxes $b_{j}$ have
to satisfy the generalized continuity equation. Note that $\int_{0}^{T}{\cal R}_{\eps}\d t$
does not induce a metric on $Q$ since the reaction part is not quadratic.

Next, we compute the tilted slope part $\cR_{\eps}^{*}(\mu,-\D\cE^{V}(\mu))$.
To do this, we introduce the relative densities $\rho^{V}$ of $\mu$
w.r.t. the stationary measure $w^{V}\d x$, i.e. $\rho_{j}^{V}=\frac{\d\mu}{w_{j}^{V}\d x}=\frac{c_{j}}{w_{j}^{V}}$,
where by \eqref{eq:StationaryMeasure} the stationary measure is $w_{j}^{V}=\frac{1}{Z}w_{i}\e^{-V_{j}}$.
Since $V\in\mathrm{C}^{1}(\Omega,\R^{2})$ and $\Omega\subset\R^{d}$
is compact, $\mu$ is absolutely continuous w.r.t. the Lebesgue measure
$\d x$ if and only if it is w.r.t. the stationary measure $w^{V}\d x$,
. Inserting $\xi=-$$\D\cE^{V}(\mu)=-(\log(c_{i}/w_{i})+V_{i})_{i=1,2}$
in the dual dissipation potential $\cR_{\eps}^{*}$, we get for the
diffusive part
\[
\cR_{\diff}^{*}(\mu,-\D\cE^{V}(\mu))=\frac{1}{2}\int_{\Omega}\sum_{j=1}^{2}\delta_{j}c_{j}|\nabla\left(\log c_{j}/w_{j}+V_{j}\right)|^{2}\d x.
\]
Using $w_{j}^{V}=\frac{1}{Z}w_{j}\e^{-V_{j}}$, a short calculation
shows $\delta_{j}c_{j}\left|\nabla\left(\log c_{j}/w_{j}+V_{j}\right)\right|^{2}=\delta_{j}w_{j}^{V}\frac{\left|\nabla\rho_{j}^{V}\right|^{2}}{\rho_{j}^{V}}.$
Hence, we have 
\[
\cR_{\diff}^{*}(\mu,-\D\cE^{V}(\mu))=\frac{1}{2}\int_{\Omega}\sum_{j=1}^{2}\delta_{j}w_{j}^{V}\frac{\left|\nabla\rho_{j}^{V}\right|^{2}}{\rho_{j}^{V}}\d x.
\]
For the reaction part, we use the identity $\mathsf{C}^{*}(\log p-\log q)=2\frac{\left(\sqrt{p}-\sqrt{q}\right)^{2}}{\sqrt{pq}}$
and get 
\begin{align*}
\cR_{\mathrm{react,\eps}}^{*}(\mu,-\D\cE^{V}(\mu)) & =2\int_{\Omega}\frac{1}{\eps}\sqrt{c_{1}c_{2}}\frac{\left(\sqrt{c_{1}/\eta_{1}w_{1}}-\sqrt{c_{2}/\eta_{2}w_{2}}\right)^{2}}{\sqrt{c_{1}c_{2}/\eta_{1}w_{1}\eta_{2}w_{2}}}\d x\\
 & =\frac{2}{\eps}\int_{\Omega}\sqrt{w_{1}^{V}w_{2}^{V}}\left(\sqrt{\rho_{1}^{V}}-\sqrt{\rho_{2}^{V}}\right)^{2}\d x.
\end{align*}
Summarizing, the total dissipation functional is
\begin{align}
\fD_{\eps}^{V}(\mu)= & \int_{0}^{T}\cR_{\eps}(\mu,\dot{\mu})+\cR_{\eps}^{*}(\mu,-V-\D\cE(\mu))\d t\label{eq:EpsDissipationFunctional}\\
= & \inf_{(c,J,b)\in\mathrm{(gCE)}}\left\{ \int_{0}^{T}\left\{ \int_{\Omega}\sum_{j=1}^{2}\widetilde{\mathsf{Q}}\left(\delta_{j}c_{j},J_{j}\right)\d x+\int_{\Omega}\widetilde{\mathsf{C}}\left(\frac{\sqrt{c_{1}c_{2}}}{\eps},b_{2}(x)\right)\d x\right\} \ \d t\right\} +\nonumber \\
 & \quad\quad+\int_{0}^{T}\left\{ \frac{1}{2}\int_{\Omega}\sum_{j=1}^{2}\delta_{j}w_{j}^{V}\frac{\left|\nabla\rho_{j}^{V}\right|^{2}}{\rho_{j}^{V}}\d x+\frac{2}{\eps}\int_{\Omega}\sqrt{w_{1}^{V}w_{2}^{V}}\left(\sqrt{\rho_{1}^{V}}-\sqrt{\rho_{2}^{V}}\right)^{2}\d x\right\} \ \d t\,.\nonumber 
\end{align}

\section{EDP-convergence result\label{sec:EDP-convergence-result}}

In this section we state the EDP-convergence result for the gradient
systems $(Q,\cE,\cR_{\eps}^{*})$ to $(Q,\cE,\cR_{\eff}^{*})$ and
discuss the properties of the effective gradient system $(Q,\cE,\cR_{\eff}^{*})$.
Since the energy $\cE$ is $\eps$-independent the major challenge
is to prove $\Gamma$-convergence of the dissipation functional $\fD_{\eps}^{V}$,
which is a functional defined on the space of trajectories in the
state space $Q$. To be mathematical precise, we first fix the functional
analytic setting.

The state space $Q=\mathrm{Prob}(\Omega\times\left\{ 1,2\right\} )$
is equipped with the $p$-Wasserstein distance $d_{\mathcal{W}_{p}}$,
where in our situation either $p=1$ or $p=2$. Recall that for any
compact euclidean subspace $E\subset\R^{k}$ the $p$-Wasserstein
distance is defined on the space of probability measures $\mathrm{Prob}(E)$
by
\[
d_{\mathcal{W}_{p}}(\mu^{1},\mu^{2})^{p}=\min_{\gamma\in\Gamma(\mu^{1},\mu^{2})}\int_{E}|x-y|{}^{p}\d\gamma(x,y),
\]
where $\Gamma(\mu^{1},\mu^{2})$ is the set of all transport plans
with marginals $\mu^{1}$ and $\mu^{2}$ (see e.g. \cite{AmGiSa05GFMS}).
The $p$-Wasserstein distance $d_{\mathcal{W}_{p}}$ metrizises the
weak{*}-topology of measures, i.e. convergence tested against continuous
functions on $E$. In the following we will consider either $E=\Omega$
or $E=\Omega\times\left\{ 1,2\right\} $.

To define the topology in the space of trajectories on $Q$, we start
very coarse, where we understand the trajectories on $Q$ as measures
in space and time. We denote the space of trajectories by $\L_{w}^{\infty}([0,T],Q)$
equipped with the weak{*}-measureability. The weak{*}-convergence
is defined as usual by 
\[
\mu^{\eps}(\cdot)\rightarrow\mu^{0}(\cdot)\ \ :\Leftrightarrow\ \ \forall i\in\left\{ 1,2\right\} ,\,\forall\phi\in\mathrm{C}^{\infty}(\Omega\times[0,T]):\ \int_{0}^{T}\!\!\!\int_{\Omega}\phi\d\mu_{i}^{\eps}(x)\d t\rightarrow\int_{0}^{T}\!\!\!\int_{\Omega}\phi\d\mu_{i}^{0}(x)\d t.
\]
A finer topology, which enables to prove compactness and evaluate
the effective dissipation functional is then given by the a priori
bounds 
\[
\sup_{\eps\in]0,1]}\fD_{\eps}^{V}(\mu^{\eps})\leq C,\quad\sup_{\eps\in]0,1]}\underset{t\in[0,T]}{\mathrm{ess\ sup}}\ \cE(\mu^{\eps}(t))\leq C.
\]
In fact, as presented in Section \ref{subsec:Compactness}, these
bounds provide that the measures $\mu^{\eps}$ have Lebesgue densities
$c^{\eps}$ which converge strongly in $\L^{1}([0,T]\times\Omega,\R_{\geq0}^{2})$.
Moreover, the limiting coarse-grained measure $\hat{\mu}^{0}=\mu_{1}^{0}+\mu_{2}^{0}$
has an representative which is absolutely continuous in time with
values in $\left(\Prob(\Omega),d_{{\cal W}_{2}}\right)$, i.e. there
is a function $m\in\L^{1}([0,T]$ such that for all $t,s\in[0,T]$
with $s\leq t$ we have
\[
d_{{\cal W}_{2}}\left(\hat{\mu}^{0}(s),\hat{\mu}^{0}(t)\right)\leq\int_{s}^{t}m(r)\d r.
\]
Each component $\mu_{i}^{0}$, $i=1,2$ is not a trajectory in the
space of probability measure, but in the space of non-negative Radon
measures. Proposition \ref{prop:AbsolutelyContinuousRepresentative}
shows that $\mu_{i}^{0}$ is absolutely continuous in time with values
in $({\cal M}_{+}(\Omega),d_{{\cal W}_{1}})$ exploiting the dual
formulation of the 1-Wasserstein distance (see e.g. \cite{Edwa11KRT}).
This compactness result is comparable to the result of Bothe and Hilhorst
\cite{BotHil02RDFRR}, where also strong convergence of solutions
$c=(c_{1},c_{2})$ is proved. In particular, similar to the space
independent situation in \cite{MieSte19ECLRS,Step19EGCGSDTS,MiPeSt20EDPCNLRS}
one cannot guarantee that $\mu^{\eps}(t)\to\mu^{0}(t)$ in $Q$ for
all times $t\in[0,T]$ as jumps in time cannot be excluded. Instead
the limit $\mu^{0}=c^{0}\,\d x$ has an absolutely continuous representative.

\subsection{Main theorem}

Let us state our main EDP-convergence result. For doing this, we define
for $V\in\mathrm{C}^{1}(\Omega,\R^{2})$ the total dissipation functional
on $\L_{w}^{\infty}([0,T],Q)$ as
\begin{align*}
\fD_{\eps}^{V}(\mu) & =\begin{cases}
\int_{0}^{T}\left\{ \cR_{\eps}(\mu,\dot{\mu})+\cR_{\eps}^{*}(\mu,-\!\D\cE^{V}(\mu))\right\} \d t, & \mu\in\mathrm{AC}([0,T],Q),\mu=c\,\d x\ \mathrm{a.e.\,in}\,[0,T]\\
\infty & \mathrm{otherwise}.
\end{cases}
\end{align*}
If $\mu=c\,\d x$ a.e. in $[0,T]$, then the dissipation functional
is given by 
\begin{align}
 & \int_{0}^{T}\cR_{\eps}(\mu,\dot{\mu})+\cR_{\eps}^{*}(\mu,-V-\D\cE(\mu))\d t\nonumber \\
 & =\inf_{(c,J,b)\in\mathrm{(gCE)}}\left\{ \int_{0}^{T}\left\{ \int_{\Omega}\sum_{j=1}^{2}\widetilde{\mathsf{Q}}(\delta_{j}c_{j},J_{j})\d x+\int_{\Omega}\widetilde{\mathsf{C}}\left(\frac{\sqrt{c_{1}c_{2}}}{\eps},b_{2}(x)\right)\d x\right\} \ \d t\right\} +\nonumber \\
 & \ \ +\int_{0}^{T}\left\{ \frac{1}{2}\int_{\Omega}\sum_{j=1}^{2}\delta_{j}w_{j}^{V}\frac{\left|\nabla\rho_{j}^{V}\right|^{2}}{\rho_{j}^{V}}\d x+\frac{2}{\eps}\int_{\Omega}\sqrt{w_{1}^{V}w_{2}^{V}}\left(\sqrt{\rho_{1}^{V}}-\sqrt{\rho_{2}^{V}}\right)^{2}\d x\ \right\} \ \d t,\label{eq:EpsDissipationFunctional-2}
\end{align}
where the infimum is taken over all Borel fluxes $J_{j}\in{\cal M}([0,T]\times\Omega,\R^{d}),b_{j}\in{\cal M}([0,T]\times\Omega,\R)$
which satisfy the generalized continuity equation (gCE) in the sense
of distributions, i.e.
\begin{align*}
\forall j & =1,2\ \forall\phi\in\mathrm{C}_{c}^{\infty}([0,T]\times\Omega):\int_{0}^{T}\!\!\int_{\Omega}\dot{\phi}c_{j}-\nabla\phi\cdot J_{j}\d x\d t=-\int_{0}^{T}\!\!\int_{\Omega}b_{j}\phi\d x\d t,\\
 & \ J\cdot\nu=0\mathrm{\ \ on}\ \partial\Omega.
\end{align*}

\begin{rem}
Strictly speaking, the functions $\widetilde{\mathsf{Q}}$ and $\widetilde{\mathsf{C}}$
are not defined for measures $J_{j}$,$b_{j}$ and hence, the formula
for the dissipation functional \eqref{eq:EpsDissipationFunctional}
is a priori not correct. In fact, introducing the related functional
as in Lemma \ref{lem:LiminfEstimateAGS}, the dissipation functional
can be expressed via the densities of $J_{j},b_{j}$. These densities
are in $\L^{1}([0,T]\times\Omega)$ as Lemma \ref{lem:L1BoundBySuperlinerarGrowth}
shows. For notational convenience, we identify the measures with their
Lebesgue densities and stick to the above expression \eqref{eq:EpsDissipationFunctional}.
In Lemma \ref{lem:RegularityFluxesAndSpaceAndLiminf}, we, in fact,
show compactness for the sequence of measures $J_{i}^{\eps}$.
\end{rem}

The main result is the $\Gamma$-convergence of $\fD_{\eps}^{V}$
to the effective dissipation functional $\fD_{0}^{V}$ which is defined
by

\begin{align*}
\fD_{0}^{V}(\mu)=\begin{cases}
\int_{0}^{T}\cR_{\mathrm{eff}}(\mu,\dot{\mu})+\cR_{\mathrm{eff}}^{*}(\mu,-\D\cE^{V}(\mu))\d t, & \mathrm{if~}\mu\in\mathrm{AC}([0,T],Q),\mu=c\,\d x\ \mathrm{a.e.\,in}\,[0,T]\\
~~\infty & \mathrm{otherwise}.
\end{cases}
\end{align*}
where
\begin{equation}
\cR_{\mathrm{eff}}^{*}(\mu,\xi)=\cR_{\mathrm{diff}}^{*}(\mu,\xi)+\chi_{\left\{ \xi_{1}=\xi_{2}\right\} }(\xi),\ \ \cR_{\mathrm{eff}}(\mu,v)=\left(\cR_{\mathrm{eff}}(\mu,\cdot)\right)^{*}(v).\label{eq:Reff}
\end{equation}

\begin{thm}
\label{thm:GammaConvergenceDissipation}Let $V\in\mathrm{C}^{1}(\Omega,\R^{2})$.
On $\L_{w}^{\infty}([0,T],Q)$, we have $\Gamma$-convergence constraint
to bounded energies of $\fD_{\eps}^{V}$, i.e. $\fD_{\eps}^{V}\xrightarrow{\Gamma_{E}}\fD_{0}^{V}$
where
\begin{align}
\fD_{0}^{V}(\mu)=\begin{cases}
\int_{0}^{T}\cR_{\mathrm{eff}}(\mu,\dot{\mu})+\cR_{\mathrm{eff}}^{*}(\mu,-V-\D\cE(\mu))\d t, & \mu\in\mathrm{AC}([0,T],Q),\mu=c\,\d x\ \mathrm{a.e.}\ [0,T]\\
~~\infty & \mathrm{otherwise}
\end{cases}\label{eq:EffectiveDissipationFunctional}
\end{align}
with
\begin{align*}
\cR_{\mathrm{eff}}^{*}(\mu,\xi) & =\cR_{\mathrm{diff}}^{*}(\mu,\xi)+\chi_{\left\{ \xi_{1}=\xi_{2}\right\} }(\xi),\\
\cR_{\mathrm{eff}}(\mu,v) & =\inf_{u+\tilde{u}=v}\left\{ \cR_{\mathrm{diff}}(\mu,\tilde{u})+\chi_{0}(u_{1}+u_{2})\right\} =\\
 & =\inf\left\{ \sum_{j=1}^{2}\int_{\Omega}\widetilde{\mathsf{Q}}(\delta_{j}c_{j},J_{j})\d x:u_{1}+u_{2}=0,\left\{ \begin{array}{c}
v_{1}=-\div J_{1}+u_{1}\\
v_{2}=-\div J_{2}+u_{2}
\end{array}\right\} \right\} \ .
\end{align*}
\end{thm}

The theorem states that the limit dissipation functional is again
of $\cR\oplus\cR^{*}$-form with an effective dissipation potential
$\cR_{\eff}^{*}$. The effective dissipation potential $\cR_{\eff}^{*}$
consists again of two terms describing the diffusion and a coupling,
which forces the chemical potential $-\D\cE^{V}$ to equilibration.
This equilibration provides the microscopic equilibria of the densities
$\rho^{V}$ defining the slow manifold of the evolution.

In the Section \ref{sec:ProofGammaConvergence}, we will present the
detailed proof of this $\Gamma$-convergence result. In this section,
we discuss the effective gradient system and its induced gradient
flow equation. As we will see the associated gradient flow equation
can be understood as an evolution equation on $Q$ and also on a smaller
state space $\hat{Q}:=\Prob(\Omega)$ of coarse-grained variables.

As an immediate consequence, Theorem \ref{thm:GammaConvergenceDissipation}
implies that $(Q,\cE,\cR_{\eps}^{*})$ EDP-converges with tilting
to $(Q,\cE,\cR_{\eff}^{*})$.
\begin{thm}
\label{thm:EDP-convergence} Let $\cR_{\eff}^{*}$ be defined by \eqref{eq:Reff}.
Then the gradient system $(Q,\cE,\cR_{\eps}^{*})$ EDP-converges with
tilting to $(Q,\cE,\cR_{\eff}^{*})$.
\end{thm}

\begin{proof}
The energy $\cE$ is $\eps$-independent and lsc. on $Q$. Hence,
it $\Gamma$-converges to itself. Theorem \ref{thm:GammaConvergenceDissipation}
implies that $\fD_{\eps}^{V}$ $\Gamma$-converges to $\fD_{0}^{V}$
and $\fD_{0}^{V}$ is of $\cR\oplus\cR^{*}$ structure, where the
effective dissipation potential $\cR_{\eff}$ is independent of the
tilts $\eta=V$. Hence, EDP-convergence with tilting is established.
\end{proof}

\subsection{Effective gradient flow equation}

In this section, we discuss the effective gradient flow equation that
is associated by the limit gradient structure $(Q,\cE,\cR_{\eff}^{*})$.
Similar to the space-independent situation in \cite{MieSte19ECLRS,MiPeSt20EDPCNLRS}
the limit gradient structure can also be equivalently understood as
a gradient structure on a smaller coarse-grained space of slow variables
$\hat{Q}$. In particular, we obtain an effective gradient flow equation
on the original state space $Q$ with a Lagrange multiplier ensuring
the projection on the slow manifold, and, moreover, an effective gradient
flow equation in coarse-grained variables. First, we discuss the effective
gradient flow equation with Lagrange multipliers, and secondly, the
coarse-grained gradient structure and its induced gradient flow equation.
Throughout the section the potential $V\in\mathrm{C}^{1}(\Omega,\R^{2})$
is fixed.

\subsubsection{Gradient flow equation with Lagrange multipliers}

For being brief, the calculations in this section are rather formal.
The effective dissipation potential $\cR_{\eff}^{*}=\cR_{\mathrm{diff}}^{*}+\chi_{\left\{ \xi_{1}=\xi_{2}\right\} }$
consists of two parts: the first describes the dissipation of the
evolution and the second provides the linear constraint of being on
the slow manifold and also the corresponding Lagrange multiplier.
The evolution equation is given by 
\[
\dot{\mu}\in\partial_{\xi}\cR_{\eff}^{*}(\mu,-V-\D\cE(\mu))=\partial_{\xi}\left\{ \cR_{\diff}^{*}(\mu,-V-\D\cE(\mu))+\chi_{\left\{ \xi_{1}=\xi_{2}\right\} }(-V-\D\cE(\mu))\right\} .
\]
Following \cite{EkeTem76CAVP}, the subdifferential of a sum is given
by the sum of the subdifferential, if one term is continuous, which
holds for the first term. For the second term, the subdifferential
of the characteristic function is only definite in its domain, i.e.
if 
\[
-V_{1}-\D\cE(\mu)_{1}=-V_{2}-\D\cE(\mu)_{2},
\]
which implies that $\mu=c\d x$ and that their densities satisfy the
relation $\frac{c_{1}}{\beta\e^{-V_{1}}}=\frac{c_{2}}{\alpha\e^{-V_{2}}}$
defining the linear slow manifold. Moreover, on its domain we have
for the subdifferential that $\partial\chi_{\left\{ \xi_{1}=\xi_{2}\right\} }={\cal M}(\Omega)\begin{pmatrix}1\\
-1
\end{pmatrix}$. Hence, we conclude that 
\begin{align*}
\dot{\mu} & \in\partial_{\xi}\left\{ \cR_{\diff}^{*}(\mu,-V-\D\cE(\mu))\right\} +{\cal M}(\Omega)\begin{pmatrix}1\\
-1
\end{pmatrix},\quad\frac{c_{1}}{\beta\e^{-V_{1}}}=\frac{c_{2}}{\alpha\e^{-V_{2}}}\ ,
\end{align*}
which implies the gradient flow equation on the slow manifold with
a Lagrange multiplier $\lambda=(\lambda_{1},\lambda_{2})$ for the
densities of the form

\begin{equation}
\begin{cases}
\dot{c_{1}}=\div\left\{ \delta_{1}\nabla c_{1}+\delta_{1}c_{1}\nabla V_{1}\right\} +\lambda_{1}(t,x)\\
\dot{c_{2}}=\div\left\{ \delta_{2}\nabla c_{2}+\delta_{2}c_{2}\nabla V_{2}\right\} +\lambda_{2}(t,x)
\end{cases},\quad\lambda_{1}+\lambda_{2}=0,\quad\frac{c_{1}}{\beta\e^{-V_{1}}}=\frac{c_{2}}{\alpha\e^{-V_{2}}}\ .\label{eq:GradientFlowEquationWithLagrangeMultiplier}
\end{equation}

\subsubsection{Coarse-grained gradient structure and its gradient flow equation}

Now, we discuss the effective gradient structure $(Q,\cE,\cR_{\eff}^{*})$
in the slow coarse-grained variables. To do this, we introduce the
coarse grained probability measure $\hat{\mu}=\mu_{1}+\mu_{2}$ on
$\Omega$ and the corresponding concentrations $\hat{c}:=c_{1}+c_{2}$.
Moreover, we define the equilibrated densities $\hat{\rho}^{V}=\rho_{1}^{V}=\rho_{2}^{V}$
and the coarse-grained stationary measure $\hat{w}^{V}=w_{1}^{V}+w_{2}^{V}$,
for which we get $\hat{c}=\rho_{1}^{V}w_{1}^{V}+\rho_{2}^{V}w_{2}^{V}=\hat{\rho}^{V}(w_{1}^{V}+w_{2}^{V})=\hat{\rho}^{V}\hat{w}^{V}$.
We introduce the coarse-grained diffusion coefficient $\hat{\delta}^{V}=\frac{\delta_{1}w_{1}^{V}+\delta_{2}w_{2}^{V}}{w_{1}^{V}+w_{2}^{V}}$.

With this notation, we may define the coarse-grained gradient structure
$(\hat{Q},\hat{\cE},\hat{\cR}^{*})$. On the state space $\hat{Q}=\Prob(\Omega)$,
we define

\begin{align}
\hat{\cR}^{*}(\hat{\mu},\hat{\xi}) & :=\cR_{\mathrm{eff}}^{*}\left(\left(\frac{w_{1}^{V}}{w_{1}^{V}+w_{2}^{V}}\hat{\mu},\frac{w_{2}^{V}}{w_{1}^{V}+w_{2}^{V}}\hat{\mu}\right),(\hat{\xi},\hat{\xi})\right)=\frac{1}{2}\int_{\Omega}\hat{\delta}^{V}|\nabla\hat{\xi}|^{2}\d\hat{\mu}\ ,\label{eq:CoarseGrainedGS}\\
\hat{\cE}(\hat{\mu}) & :=\cE^{V}\left(\frac{w_{1}^{V}}{w_{1}^{V}+w_{2}^{V}}\hat{\mu},\frac{w_{2}^{V}}{w_{1}^{V}+w_{2}^{V}}\hat{\mu}\right)\ .\nonumber 
\end{align}
Introducing the coarse-grained potential $\hat{V}=-\log(w_{1}^{V}+w_{2}^{V})-\log Z=-\log(\hat{w}^{V})-\log Z=-\log\left(w_{1}\e^{-V_{1}}+w_{2}\e^{-V_{2}}\right)$,
for which the exponential is given by the weighted arithmetic mean
of the exponentials $\e^{-V_{1}}$ and $\e^{-V_{2}}$, i.e. $\e^{-\hat{V}}=w_{1}\e^{-V_{1}}+w_{2}\e^{-V_{2}}$
(we used that $w_{1}+w_{2}=1$). Easy calculations show that the energy
has for the explicit form
\[
\hat{\cE}(\hat{\mu})=\int_{\Omega}\hat{\mu}\log\hat{\mu}+\hat{\mu}\hat{V}\d x.
\]
The coarse-grained dissipation functional is defined by 
\[
\hat{\fD}(\hat{\mu})=\int_{0}^{T}\hat{\cR}(\hat{\mu},\dot{\hat{\mu}})+\hat{\cR}^{*}(\hat{\mu},-\D\hat{\cE}(\hat{\mu}))\d t,
\]
which incorporates the tilt via the coarse-grained variables. Note,
that the coarse-grained dissipation potential $\hat{\cR}^{*}$ depends
explicitly on the tilt $V$ via the diffusion coefficient $\delta^{V}$.
This is not a contradiction to tilt-EDP convergence (Theorem \ref{thm:EDP-convergence}),
because in original variables the effective dissipation potential
\eqref{eq:EffectiveDissipationFunctional} is indeed independent of
the tilts. The tilts dependence of $\hat{\cR}^{*}$ originates from
the energy and tilt dependent slow manifold.

To relate the dissipation functional $\fD_{0}^{V}$ with the coarse
grained dissipation functional $\hat{\fD}$, we first show that also
an equilibration of the fluxes occurs. To do this the following convexity
property is important.
\begin{lem}
\label{lem:Convexity}Let $X$ a separable and reflexive Banach space
and let $\mathsf{F}:X\rightarrow\R_{\infty}$ be convex and lsc. Then
for the function $\widetilde{\mathsf{F}}:[0,\infty[\times X\to\R_{\infty}$,
we have
\begin{align*}
\widetilde{\mathsf{F}}\left(\sum_{i=1}^{I}a_{i},\sum_{i=1}^{I}x_{i}\right)\leq\sum_{i=1}^{I}\widetilde{\mathsf{F}}(a_{i},x_{i}).
\end{align*}
If $\mathsf{F}$ is strictly convex then equality holds if and only
if $(a_{i},x_{i})=(0,0)$ whenever $a_{i}=0$ or $x_{i}/a_{i}=x_{j}/a_{j}$
whenever $a_{i},a_{j}>0$. Moreover, if $\mathsf{F}(0)=0$, we have
the following monotonicity property
\[
\widetilde{\mathsf{F}}(a_{1},x)\leq\widetilde{\mathsf{F}}(a_{2},x),\ \ \mathrm{if}\ a_{1}\geq a_{2}.
\]
\end{lem}

\begin{proof}
Let pairs $(a_{i},x_{i})$ for $i=1,\dots,I$ be given. If $a_{i}=0$,
then either $x_{i}=0$ and the claim has to be shown for $I-1$ -number
of pairs, or $x_{i}\neq0$ and the right-hand side is $\infty$ meaning
that the claim is trivial. So let us assume that $a_{i}>0$ for all
$i=1\dots,I$. Then $\widetilde{\mathsf{F}}(a_{i},x_{i})=a_{i}\mathsf{F}(x_{i}/a_{i})$
and the claim is equivalent to 
\begin{align*}
\sum_{i=1}^{I}\frac{a_{i}}{\sum_{i=1}^{I}a_{i}}\mathsf{F}(x_{i}/a_{i})\geq\mathsf{F}\left(\sum_{i=1}^{I}\frac{a_{i}}{\sum_{i=1}^{I}a_{i}}\frac{x_{i}}{a_{i}}\right)=\mathsf{F}\left(\frac{\sum_{i=1}^{I}x{}_{i}}{\sum_{i=1}^{I}a_{i}}\right),
\end{align*}
which holds since $\mathsf{F}$ is convex. If $F$ is strictly convex
then we immediately observe that whenever $a_{i},a_{j}>0$ we have
$\tfrac{x_{i}}{a_{i}}=\tfrac{x_{j}}{a_{j}},$ and whenever $a_{i}=0$
that also $x_{i}=0$.

To see the monotonicity property, we observe that $\widetilde{\mathsf{F}}(a,0)=0$
for all $a\geq0$. Hence, we have
\[
\widetilde{\mathsf{F}}(a_{1}+a_{2},x)\leq\widetilde{\mathsf{F}}(a_{1},x)+\widetilde{\mathsf{F}}(a_{2},0)=\widetilde{\mathsf{F}}(a_{1},x),
\]
which proves the claim.
\end{proof}
Recalling formula \eqref{eq:EffectiveDissipationFunctional} of the
effective dissipation functional $\fD_{0}^{V}$ and using the above
lemma, we observe that the velocity part of the dissipation functional
$\fD_{0}^{V}$ can now be estimated. In particular, we will see that
the limit dissipation functional $\fD_{0}^{V}$ can be equivalently
expressed in coarse-grained variables $(\hat{\mu},\hat{J})$ by using
that an equilibration of concentrations also provides an equilibration
of the corresponding fluxes. In the reconstruction strategy in Section
\ref{subsec:ConstructionRecoverySequence} this equilibration is explicitly
used (see \ref{eq:ReconstructionFluxes} and \eqref{eq:Reconstruction}).
\begin{prop}
\label{prop:DissipationFunctionalInCoarseGrainedVariables}Let $\mu\in\mathrm{AC}([0,T],Q)$
with $\fD_{0}^{V}(\mu)<\infty$ and $\mathrm{ess\,\sup}_{t\in[0,T]}\cE(\mu(t))<\infty$.
Then the following holds:
\begin{enumerate}
\item We have $\fD_{0}^{V}(\mu)=\hat{\fD}(\hat{\mu})$ where $\hat{\mu}=\mu_{1}+\mu_{2}$
and
\begin{align*}
\mathsf{\hat{\fD}}(\hat{\mu}) & =\int_{0}^{T}\hat{\cR}(\hat{\mu},\dot{\hat{\mu}})+\hat{\cR}^{*}(\hat{\mu},-\D\hat{\cE}(\hat{\mu}))\ \d t\\
 & =\inf_{\hat{J}:\dot{\hat{\mu}}+\div\hat{J}=0}\left\{ \int_{0}^{T}\int_{\Omega}\widetilde{\mathsf{Q}}(\hat{\delta}^{V}\hat{c},\hat{J})+\frac{\hat{\delta}\hat{w}^{V}}{2}\frac{|\nabla\hat{\rho}^{V}|^{2}}{\hat{\rho}^{V}}\ \d x\d t\right\} .
\end{align*}
\item The chain-rule holds for $[0,T]\ni t\mapsto\hat{\cE}(\hat{\mu}(t))\in\R$,
i.e. we have 
\[
\frac{\d}{\d t}\hat{\cE}(\hat{\mu}(t))=\langle\D\hat{\cE}(\hat{\mu}(t)),\dot{\hat{\mu}}(t)\rangle\ .
\]
\item The gradient flow equation of the gradient system $(\hat{Q},\hat{\cE},\hat{\cR}^{*})$
is given by
\begin{equation}
\dot{\hat{c}}=-\div\left(\hat{\delta}^{V}\hat{c}\,\nabla\left(-\D\hat{\cE}(\hat{\mu})\right)\right)=\div\left(\hat{\delta}^{V}\nabla\hat{c}+\hat{\delta}^{V}\hat{c}\nabla\hat{V}\right),\label{eq:CoarseGrainedGradienFlowEquation}
\end{equation}
with the potential $\hat{V}=-\log\hat{w}^{V}$ and stationary measure
$\hat{w}^{V}$.
\end{enumerate}
\end{prop}

Equation \eqref{eq:CoarseGrainedGradienFlowEquation} shows that the
coarse-grained gradient flow equation induced by $(\hat{Q},\hat{\cE},\hat{\cR}^{*})$
is a drift-diffusion equation of the coarse-grained concentration
$\hat{c}$ with mixed diffusion constant $\hat{\delta}^{V}$. In particular,
in the tilt free case we have $\hat{\delta}^{V=\mathrm{const}}=\frac{\beta\delta_{1}+\alpha\delta_{2}}{\alpha+\beta}$,
and we recover the result of \cite{BotHil02RDFRR}.
\begin{proof}
To prove Part 1, we first observe that the bounded energy and dissipation
for the trajectory $\mu$ implies that we have $\frac{c_{1}}{w_{1}^{V}}=\frac{c_{2}}{w_{2}^{V}}$
a.e. in $[0,T]\times\Omega$. Using $\hat{c}=c_{1}+c_{2}$ for the
densities, we get $\hat{c}=\frac{w_{1}^{V}+w_{2}^{V}}{w_{1}^{V}}c_{1}=\frac{w_{1}^{V}+w_{2}^{V}}{w_{2}^{V}}c_{2}$.
The Fisher information ${\cal S}_{0}^{V}(\mu):=\cR_{\mathrm{eff}}^{*}(\mu,-V-\D\cE(\mu))$
has the form
\begin{align}
{\cal S}_{0}^{V}(\mu) & =\frac{1}{2}\int_{\Omega}\sum_{j=1}^{2}\delta_{j}w_{j}^{V}\frac{\left|\nabla\rho_{j}^{V}\right|^{2}}{\rho_{j}^{V}}\d x=\frac{1}{2}\int_{\Omega}\left(\frac{\delta_{1}w_{1}^{V}+\delta_{2}w_{2}^{V}}{w_{1}^{V}+w_{2}^{V}}\right)\left(w_{1}^{V}+w_{2}^{V}\right)\frac{|\nabla\hat{\rho}^{V}|^{2}}{\hat{\rho}^{V}}\d x\label{eq:LimitFisherInformation}\\
 & =\frac{1}{2}\int_{\Omega}\hat{\delta}^{V}\hat{w}^{V}\frac{|\nabla\hat{\rho}^{V}|^{2}}{\hat{\rho}^{V}}\d x=\hat{\cR}^{*}(\hat{\mu},-\D\hat{\cE}(\hat{\mu})).
\end{align}

Lemma \ref{lem:Convexity} provides that also an equilibration of
the fluxes occurs. Indeed, defining the coarse-grained flux $\hat{J}=J_{1}+J_{2}$,
we conclude
\begin{align}
\frac{|J_{1}|^{2}}{\delta_{1}c_{1}}+\frac{|J_{2}|^{2}}{\delta_{2}c_{2}}\geq\frac{|J_{1}+J_{2}|^{2}}{\delta_{1}c_{1}+\delta_{2}c_{2}}=\frac{|\hat{J}|^{2}}{\frac{\delta_{1}w_{1}^{V}+\delta_{2}w_{2}^{V}}{w_{1}^{V}+w_{2}^{V}}\hat{c}}=\frac{|\hat{J}|^{2}}{\hat{\delta}\hat{c}},\label{eq:EstimateCoarseGrainedFluxes}
\end{align}
where $\hat{\delta}^{V}:=\frac{\delta_{1}w_{1}^{V}+\delta_{2}w_{2}^{V}}{w_{1}^{V}+w_{2}^{V}}$.
Equality holds if and only if $(J_{1},c_{1})=0$ or $(J_{2},c_{2})=0$
or $J_{1}/\delta_{1}c_{1}=J_{2}/\delta_{2}c_{2}=\hat{J}/\hat{\delta}^{V}\hat{c}$.
The last condition is equivalent to 
\begin{equation}
\hat{J}=\frac{\delta_{1}w_{1}^{V}}{\delta_{1}w_{1}^{V}+\delta_{2}w_{2}^{V}}J_{1}=\frac{\delta_{2}w_{2}^{V}}{\delta_{1}w_{1}^{V}+\delta_{2}w_{2}^{V}}J_{2}\ ,\label{eq:ReconstructionFluxes}
\end{equation}
which provides an explicit formula for the coarse-grained diffusion
flux.

For the dissipation functional that means 
\begin{align}
\fD_{0}^{V}(\mu) & =\inf_{(c,J,b)\in\mathrm{(gCE)}}\int_{0}^{T}\left\{ \int_{\Omega}\sum_{j=1}^{2}\widetilde{\mathsf{Q}}(\delta_{j}c_{j},J_{j})\d x+\frac{1}{2}\int_{\Omega}\sum_{j=1}^{2}\delta_{j}w_{j}^{V}\frac{|\nabla\rho_{j}|^{2}}{\rho_{j}}\d x\right\} \d t\nonumber \\
 & \geq\inf_{\dot{\hat{c}}+\div\hat{J}=0}\int_{0}^{T}\left\{ \int_{\Omega}\widetilde{\mathsf{Q}}\left(\hat{\delta}\hat{c},\hat{J}\right)\d x+\frac{1}{2}\int_{\Omega}\hat{\delta}\hat{w}^{V}\frac{|\nabla\hat{\rho}|^{2}}{\hat{\rho}}\d x\right\} \d t.\label{eq:CoarseGrainingEstimateDissipationFunctional}
\end{align}
To prove equality, we first observe that $\hat{J},J_{1},J_{2}$ satisfy
the same boundary conditions. Moreover, the explicitly derived reaction
flux $b_{1},b_{2}$ from \eqref{eq:Reconstruction} shows that the
reconstructed fluxes $(J_{1},J_{2})$ from coarse-grained flux $\hat{J}$
is admissible. Hence, we obtain equality
\[
\fD_{0}^{V}(\mu)=\inf_{\dot{\hat{c}}+\div\hat{J}=0}\int_{0}^{T}\left\{ \frac{1}{2}\int_{\Omega}\mathsf{Q}\left(\hat{\delta}\hat{c},\hat{J}\right)\d x+\frac{1}{2}\int_{\Omega}\hat{\delta}\hat{w}^{V}\frac{|\nabla\hat{\rho}|^{2}}{\hat{\rho}}\d x\right\} \d t,
\]
which proves the first part.

For the chain-rule in Part 2, we refer to the proof \cite[Lem 4.8]{FreLie19EDTS}
since we consider the pure diffusive situation. The proof uses a time-regularization
argument and convexity of the Fisher-information following the ideas
of \cite[Prop. 2.4]{MiRoSa13NADN}.

For Part 3, we compute the evolution equation that is induced by the
gradient system is $(\hat{Q},\hat{\cE},\hat{\cR}^{*})$. We have
\begin{align*}
\partial_{\hat{\xi}}\hat{\cR}^{*}(\hat{\mu},\hat{\xi}) & =-\div\left(\hat{\delta}^{V}\hat{c}\nabla\hat{\xi}\right),\\
\D\hat{\cE}(\hat{\mu}) & =\log\hat{\mu}+1-\log\hat{w}^{V}-\log Z,\\
\nabla\left(-\D\hat{\cE}(\hat{\mu})\right) & =-\frac{\nabla\hat{\mu}}{\hat{\mu}}+\frac{\nabla\hat{w}^{V}}{\hat{w}^{V}}=-\frac{\nabla\hat{\mu}}{\hat{\mu}}-\nabla\hat{V},
\end{align*}
which results in
\[
\hat{c}=-\div\left(\hat{\delta}^{V}\hat{c}\,\nabla\left(-\D\mathsf{E}(\hat{\mu})\right)\right)=\div\left(\hat{\delta}^{V}\nabla\hat{c}+\hat{\delta}^{V}\hat{c}\nabla\hat{V}\right).
\]
\end{proof}
Note that the coarse-grained gradient flow equation \eqref{eq:CoarseGrainedGradienFlowEquation}
is equivalent to the gradient flow equation with Lagrange multipliers
\eqref{eq:GradientFlowEquationWithLagrangeMultiplier}. Indeed, adding
both equations in \eqref{eq:GradientFlowEquationWithLagrangeMultiplier}
together and using that the original concentrations can be expressed
by the coarse-grained concentrations via
\begin{equation}
c_{i}=\frac{w_{i}\e^{-V_{i}}}{w_{1}\e^{-V_{1}}+w_{2}\e^{-V_{2}}}\hat{c},\label{eq:RelationCoarseGrainedConcentrationsOriginalConcentrations}
\end{equation}
the coarse-grained gradient flow equation \eqref{eq:CoarseGrainedGradienFlowEquation}
with the drift term $\hat{\delta}^{V}\hat{c}\nabla\hat{V}$ can be
readily derived. Conversely, using \eqref{eq:RelationCoarseGrainedConcentrationsOriginalConcentrations},
we see that $c=(c_{1},c_{2})$ are on the slow manifold and satisfy
\eqref{eq:GradientFlowEquationWithLagrangeMultiplier}. The corresponding
Lagrange multipliers $\lambda=(\lambda_{1},\lambda_{2})$ can be explicitly
calculated. Introducing the difference of the diffusion constants
$\overline{\delta}=\delta_{1}-\delta_{2}$ and the potentials $\overline{V}=V_{1}-V_{2}$,
we have 
\begin{align*}
\lambda_{1} & =\frac{w_{2}\e^{-V_{2}}}{w_{1}\e^{-V_{1}}+w_{2}\e^{-V_{2}}}\left(-\overline{\delta}\Delta c_{1}+\left(\delta_{2}\nabla\overline{V}-\overline{\delta}\nabla V_{1}\right)\cdot\nabla c_{1}+c_{1}\left\{ \delta_{2}\nabla\overline{V}\,\nabla V_{1}-\overline{\delta}\Delta V_{1}\right\} \right)\\
\lambda_{2} & =\frac{w_{1}\e^{-V_{1}}}{w_{1}\e^{-V_{1}}+w_{2}\e^{-V_{2}}}\left(\overline{\delta}\Delta c_{2}+\left(-\delta_{1}\nabla\overline{V}+\overline{\delta}\nabla V_{2}\right)\cdot\nabla c_{2}+c_{2}\left\{ -\delta_{1}\nabla\overline{V}\,\nabla V_{2}+\overline{\delta}\Delta V_{2}\right\} \right).
\end{align*}
We observe that the Lagrange multiplier $\lambda_{i}$ has the same
regularity as the right-hand side of the evolution of $c_{i}$. Moreover,
both evolution equations are completely uncoupled but contain a linear
annihilation/creation term, which depends on the potential $V=(V_{1},V_{2})$
and the diffusion coefficient $\delta=\left(\delta_{1},\delta_{2}\right)$.
A lengthy calculation shows that indeed we have $\lambda_{1}+\lambda_{2}=0$.

\section{Proof of $\Gamma$-convergence\label{sec:ProofGammaConvergence}}

In this section we prove the $\Gamma$-convergence result of Theorem
\ref{thm:GammaConvergenceDissipation}. As usual, we prove $\Gamma$-convergence
in three steps: First deriving compactness, secondly establishing
the liminf-estimate by exploiting the compactness, thirdly constructing
the recovery sequence for the limsup-estimate.

In the following the next lemma will be useful.
\begin{lem}
\label{lem:L1BoundBySuperlinerarGrowth}Let $\mathsf{F}:\R^{k}\rightarrow[0,\infty[$
be a convex, lsc. function of superlinear growth, i.e. $\mathsf{F}(r)/r\to\infty$
as $r\to\infty$. Then there is a constant $k_{\mathsf{F}}>0$ such
that for any measurable functions $W:\Omega\rightarrow\R^{k}$ and
$\rho:\Omega\rightarrow\R_{\geq0}$ it holds 
\[
\int_{\Omega}|W|\d x\leq\int_{\Omega}\widetilde{\mathsf{F}}(\rho,W)\d x+k_{\mathsf{F}}\int_{\Omega}\rho\d x\ .
\]
\end{lem}

\begin{proof}
Let $W$ and $\rho$ be given. We define three measurable subsets
of $\Omega$:
\[
\Omega_{0}=\{x:\rho(x)=0\},\ \Omega_{1}=\{x:\rho\neq0,\tfrac{1}{\rho}|W|\leq\mathsf{F}(\tfrac{1}{\rho}W)\},\ \Omega_{2}=\{x:\rho\neq0,\tfrac{1}{\rho}|W|>\mathsf{F}(\tfrac{1}{\rho}W)\}\ .
\]
Since $\mathsf{F}$ is superlinear, there is a constant $k_{\mathsf{F}}>0$
such that on $\Omega_{2}$ it holds $W/\rho\leq k_{\mathsf{F}}$.
Hence we can estimate
\begin{align*}
\int_{\Omega}|W|\d x & \leq\int_{\Omega_{0}}|W|\d x+\int_{\Omega_{1}}\frac{|W|}{\rho}\rho\d x+\int_{\Omega_{2}}\frac{|W|}{\rho}\rho\d x\\
 & \leq\int_{\Omega_{0}}\widetilde{\mathsf{F}}(\rho,W)\d x+\int_{\Omega_{1}}\mathsf{F}(\tfrac{1}{\rho}W)\rho\d x+k_{\mathsf{F}}\int_{\Omega_{2}}\rho\d x\\
 & \leq\int_{\Omega}\widetilde{\mathsf{F}}(\rho,W)\d x+k_{\mathsf{F}}\int_{\Omega}\rho\d x\ .
\end{align*}
\end{proof}
Moreover, we need the following classical lemma. It guarantees the
necessary regularity for the limits, and moreover, it provides the
desired liminf-estimate.
\begin{lem}[AGS05-Lemma 9.4.3, \cite{AmGiSa05GFMS}]
 \label{lem:LiminfEstimateAGS}Let $F:[0,\infty[\rightarrow[0,\infty]$
be a proper, lsc, convex function with superlinear growth. We define
the related functional
\[
\mathcal{F}(\mu,\gamma)=\begin{cases}
\int_{A}F(\frac{\d\mu}{\d\gamma})\d\gamma, & \mathrm{if}\ \ \mu\ll\gamma,\\
\infty, & \mathrm{otherwise}.
\end{cases}
\]
Let $\mu^{\eps},\gamma^{\eps}\in\mathrm{Prob}(A)$ be two sequences
with $\mu^{\eps}\wstarlim\mu^{0}$ and $\gamma^{\eps}\wstarlim\gamma^{0}$.
Then
\[
\liminf_{\eps\rightarrow0}\mathcal{F}(\mu^{\eps},\gamma^{\eps})\geq\mathcal{F}(\mu^{0},\gamma^{0}).
\]
In particular, if the left-hand side is finite then for the limits
it holds $\mu^{0}\ll\gamma^{0}.$
\end{lem}

\subsection{Compactness\label{subsec:Compactness}}

In this section, we derive the required compactness for proving the
liminf-estimate in Section \ref{subsec:Liminf-estimate}.

Recall that for given potential $V\in\mathrm{C}^{1}(\Omega,\R^{2})$
the dissipation functional $\fD_{\eps}^{V}$ is defined on the space
of trajectories equipped with the weak topology, i.e. $\mu^{\eps}\rightarrow\mu^{0}\in\L_{w}^{\infty}([0,T],Q)$
if and only if it holds

\[
\forall i\in\left\{ 1,2\right\} ,\,\forall\phi\in\mathrm{C}_{0}^{\infty}(\Omega\times[0,T]):\ \ \int_{0}^{T}\int_{\Omega}\phi\d\mu_{i}^{\eps}(x)\d t\rightarrow\int_{0}^{T}\int_{\Omega}\phi\d\mu_{i}^{0}(x)\d t\ .
\]
In the following we want to derive compactness for a sequence $(\mu^{\eps})_{\eps>0}$
of trajectories, satisfying the a priori bounds
\begin{align}
\sup_{\eps>0}\ \underset{t\in[0,T]}{\mathrm{ess\,sup}\ }\cE(\mu^{\eps}(t)) & \leq C,\quad\quad\sup_{\eps>0}\fD_{\eps}^{V}(\mu^{\eps})\leq C,\label{eq:AprioriBoundsDissipationFunctional}
\end{align}
where the total dissipation functional is
\begin{align*}
\fD_{\eps}^{V}(\mu) & =\begin{cases}
\int_{0}^{T}\cR_{\eps}(\mu,\dot{\mu})+\cR_{\eps}^{*}(\mu,-\D\cE^{V}(\mu))\d t, & \mu\in\mathrm{AC}([0,T],Q),\mu=c\,\d x\ \mathrm{a.e.\ in\ }[0,T]\\
\infty & \mathrm{otherwise},
\end{cases}
\end{align*}
and for $\mu=c\,\d x$ we have
\begin{align}
 & \int_{0}^{T}\cR_{\eps}(\mu,\dot{\mu})+\cR_{\eps}^{*}(\mu,-V-\D\cE(\mu))\d t\nonumber \\
 & =\inf_{(c,J,b)\in\mathrm{(gCE)}}\left\{ \int_{0}^{T}\left\{ \int_{\Omega}\sum_{j=1}^{2}\widetilde{\mathsf{Q}}\left(\delta_{j}c_{j},J_{j}\right)\d x+\int_{\Omega}\widetilde{\mathsf{C}}\left(\frac{\sqrt{c_{1}c_{2}}}{\eps},b_{2}(x)\right)\d x\right\} \ \d t\right\} +,\nonumber \\
 & +\int_{0}^{T}\frac{1}{2}\int_{\Omega}\sum_{j=1}^{2}\delta_{j}w_{j}^{V}\frac{\left|\nabla\rho_{j}^{V}\right|^{2}}{\rho_{j}^{V}}\d x+\frac{2}{\eps}\int_{\Omega}\sqrt{w_{1}^{V}w_{2}^{V}}\left(\sqrt{\rho_{1}^{V}}-\sqrt{\rho_{2}^{V}}\right)^{2}\d x\ \ \d t\ .\label{eq:EpsDissipationFunctional-1}
\end{align}
Using the bound of the dissipation functional $\fD_{\eps}^{V}(\mu^{\eps})\leq C$,
we conclude that there are diffusive fluxes $J^{\eps}=(J_{1}^{\eps},J_{2}^{\eps})$
and reaction fluxes $b^{\eps}=(b_{1}^{\eps},b_{2}^{\eps})$ such that
$J_{i}^{\eps}\in{\cal M}(\Omega,\R^{d})$, $b_{i}^{\eps}\in{\cal M}(\Omega,\R)$
and $(\mu^{\eps},J^{\eps},b^{\eps})$ satisfies the continuity equation
\[
(c,J,b)\in\mathrm{{(gCE)}}\ \ \Leftrightarrow\ \ \left\{ b_{1}+b_{2}=0\ \mathrm{and}\ \left\{ \begin{array}{c}
\dot{c}_{1}=-\div J_{1}+b_{1}\\
\dot{c}_{2}=-\div J_{2}+b_{2}
\end{array}\right\} \right\} .
\]
 Moreover, we get bounds:
\begin{align*}
i=1,2 & :\ \ \int_{0}^{T}\int_{\Omega}\frac{|\nabla\rho_{i}^{V,\eps}|^{2}}{\rho_{i}^{V,\eps}}\d x\d t\leq C,\quad\int_{0}^{T}\int_{\Omega}\widetilde{\mathsf{Q}}\left(c_{i}^{\eps},J_{i}^{\eps}\right)\d x\d t\leq C,\\
 & \int_{0}^{T}\int_{\Omega}\widetilde{\mathsf{C}}\left(\frac{\sqrt{c_{1}^{\eps}c_{2}^{\eps}}}{\eps},b_{2}^{\eps}(x)\right)\d x\d t\leq C,\quad\frac{1}{\eps}\int_{0}^{T}\int_{\Omega}\left(\sqrt{\rho_{1}^{V,\eps}}-\sqrt{\rho_{2}^{V,\eps}}\right)^{2}\d x\d t\leq C\ .
\end{align*}

\begin{rem}
Following \cite{Mani07PRUR} a distributional solution $(\mu,J,B)$
of the generalized continuity equation $\dot{\mu}=-\div J+B$ satisfying
$\int_{0}^{T}\int_{\Omega}|B|+|J|\d x<\infty$ can be assumed to be
absolutely continuous. The bounds can be obtained easily using Lemma
\ref{lem:L1BoundBySuperlinerarGrowth} for fixed $\eps>0$.
\end{rem}

Although the functional is convex in the concentration $c$ and in
the fluxes $J$ and $b$, weak convergence would be sufficient to
prove a liminf-estimate using a Joffe-type argument. But, comparing
the situation with the evolution equation, we aim in proving even
strong convergence for the densities $c^{\eps}\to c^{0}$ in $\L^{1}([0,T]\times\Omega,\R_{\geq0}^{2})$.
This is done in two steps: First, compactness of coarse-grained variables,
and secondly, convergence towards the slow manifold is shown, which
together implies strong compactness. This strategy has successfully
been applied already in the space-independent case in \cite{MieSte19ECLRS,MiPeSt20EDPCNLRS}.
Moreover, we show that the limit trajectory $\mu^{0}=c^{0}\d x$ has
a representative which is in $\mathrm{AC}([0,T],Q)$. Note that it
is not possible to prove pointwise convergence $\mu^{\eps}(t)\wstarlim\mu^{0}(t)$
for all $t\in[0,T]$. Instead, pointwise convergence is only shown
for the coarse-grained variables $\hat{\mu}^{\eps}:=\mu_{1}^{\eps}+\mu_{2}^{\eps}$.

First, we derive weak compactness in space-time, which immediately
follows from the uniform bound in $\eps$ and time on the energy.
\begin{lem}[Very Weak compactness in space-time]
Let $(\mu^{\eps})_{\eps>0}$, $\mu^{\eps}\in\L_{w}^{\infty}([0,T],Q)$
satisfy $\sup_{\eps>0}\ \underset{t\in[0,T]}{\mathrm{ess\,sup}\ }\cE(\mu^{\eps}(t))\leq C$.
Then for a.e. $t\in[0,T]$ the measure $\mu^{\eps}(t,\cdot)$ has
a Lebesgue density $c^{\eps}(t,\cdot)$. Moreover, there is a subsequence
(not relabeled), such that their densities $c^{\eps}$ are uniformly
integrable in $\Omega\times[0,T]\times\left\{ 1,2\right\} $ and hence,
$c_{i}^{\eps}$ converges weakly in $\L^{1}([0,T]\times\Omega)$ to
$c_{i}^{0}$ for $i=1,2$.
\end{lem}

\begin{proof}
The bound on the energy implies that a.e. $t\in[0,T]$ the measure
$\mu^{\eps}(t,\cdot)$ has a Lebesgue density $c^{\eps}(t,\cdot)$.
Moreover, the functional $\mu\mapsto\int_{0}^{T}{\cal E}(\mu)\d t$
is superlinear and convex. Hence, it follows by the Theorem of de
Valleé -Poussin that $\mu^{\eps}$ are uniformly integrable and hence,
$\mu_{i}^{\eps}$ converges weakly in $\L^{1}([0,T]\times\Omega)$
to $\mu_{i}^{0}$.
\end{proof}
In the following, we are going to derive compactness for the concentrations
$c_{i}^{\eps}$ and the diffusive fluxes $J_{i}^{\eps}$. It is not
possible to get compactness for the fast reaction flux $b_{2}^{\eps}$
by bounding the dissipation functional. In particular, pointwise convergence
for the measures $\mu^{\eps}(t)$ cannot be achieved.
\begin{rem}
\label{rem:NoBoundOnFastReactionFluxes}To see that compactness for
the fast reaction flux $b_{2}^{\eps}$ is not possible obtain, we
set $\rho^{\eps}=1$ constant in $[0,T]\times\Omega\times\left\{ 1,2\right\} $
and $b_{2}^{\eps}=b^{\eps}$ constant in $[0,T]\times\Omega$. Then,
a bound means on the dissipation functional implies a bound 
\[
\infty>\int_{0}^{T}\int_{\Omega}\mathsf{C}\left(\frac{\sqrt{c_{1}^{\eps}c_{2}^{\eps}}}{\eps},b_{2}^{\eps}(x)\right)\d x\approx\mathsf{C}(\frac{1}{\eps},b^{\eps})=\frac{1}{\eps}\mathsf{C}(\eps b^{\eps})\approx|b^{\eps}|\log(\eps|b^{\eps}|+1).
\]
Setting $b^{\eps}=-\log\eps$, we easily see that $|b^{\eps}|\log(\eps|b^{\eps}|+1)\rightarrow0$
as $\eps\rightarrow0$, however, $b^{\eps}\rightarrow\infty$. Hence,
it is not possible to obtain compactness for the fast reaction flux
$b_{i}^{\eps}$. Later in Lemma \ref{lem:BinLCimpliesCepstozero}
the ``converse'' statement is proved: If $\int\!\!\int\mathsf{C}(b)\d x\d t<\infty$
then $\int\!\!\int\mathsf{C}(\frac{1}{\eps},b)\d x\d t\rightarrow0$.
\end{rem}

Next, we are going to derive time-regularity for the sequence $(\mu^{\eps})_{\eps>0}$
in proving compactness for the coarse-grained trajectories $\hat{\mu}^{\eps}=\mu_{1}^{\eps}+\mu_{2}^{\eps}$.
In particular, we are able to prove pointwise convergence in time.
\begin{lem}
[Time Regularity of $\mu^{\eps}$]\label{LemmaRegInTime} Let $(\mu^{\eps})_{\eps>0}$,
$\mu^{\eps}\in\L_{w}^{\infty}([0,T],Q)$ satisfying the a priori bounds
\eqref{eq:AprioriBoundsDissipationFunctional}. Then the curves $t\mapsto\hat{\mu}^{\eps}:=\mu_{1}^{\eps}(t)+\mu_{2}^{\eps}(t)$
have $\eps$-uniform bounded total variation in the space $\mathrm{Prob}(\Omega)$
equipped with the 1-Wasserstein distance, i.e.
\[
\|\hat{\mu}^{\eps}\|_{TV}:=\sup\left\{ \sum_{k=1}^{K}{\cal W}_{1}(\hat{\mu}^{\eps}(t_{k}),\hat{\mu}^{\eps}(t_{k-1}))\ :\ 0=t_{0}<\cdots<t_{k}<\cdots<t_{K}=T\right\} .
\]
 In particular, by Helly's selection principle, we conclude pointwise
convergence $\hat{\mu}^{\eps}(t):=\mu_{1}^{\eps}(t)+\mu_{2}^{\eps}(t)\wstarlim\hat{\mu}^{0}(t):=\mu_{1}^{0}(t)+\mu_{2}^{0}(t)$
for all $t\in[0,T]$ in $\mathrm{Prob}(\Omega)$ along a suitable
subsequence.
\end{lem}

\begin{proof}
We exploit the dual formulation of the $\L^{1}$-Wasserstein distance,
i.e. integrating against Lipschitz functions (see e.g. \cite{AmGiSa05GFMS}).
Take $\phi\in\mathrm{C}^{1}(\Omega\times\left\{ 1,2\right\} )$ with
$\|\phi\|_{\W^{1,\infty}(\Omega)}\leq1$ and $\phi=(\hat{\phi},\hat{\phi})$.
Using the continuity equations $\left\{ b_{1}+b_{2}=0\ \mathrm{and}\ \left\{ \begin{array}{c}
\dot{c}_{1}=-\div J_{1}+b_{1}\\
\dot{c}_{2}=-\div J_{2}+b_{2}
\end{array}\right\} \right\} $, we conclude for all $[t_{1},t_{2}]\subset[0,T]$ that 
\begin{align*}
\int_{\Omega}\phi\cdot(\d\mu^{\eps}(t_{2})-\d\mu^{\eps}(t_{1})) & =\int_{t_{1}}^{t_{2}}\langle\phi,\dot{\mu}^{\eps}\rangle\d t\\
 & \leq\int_{t_{1}}^{t_{2}}\!\!\int_{\Omega}\nabla\hat{\phi}\cdot\left(J_{1}^{\eps}+J_{2}^{\eps}\right)\ \d x\ \d t\leq\int_{t_{1}}^{t_{2}}\!\!\int_{\Omega}\sum_{i=1}^{2}|J_{i}^{\eps}|\ \d x\ \d t.
\end{align*}
By the bound on the dissipation functional, we obtain $\eps$-uniform
bounds on the term $\int_{t_{1}}^{t_{2}}\int_{\Omega}\mathsf{Q}(\delta_{j}c_{j},J_{j})\d x\d t$.
Moreover, we have $\int_{\Omega}\rho_{j}^{\eps}\leq1$ for all $t\in[0,T]$.
Hence by Lemma \ref{lem:L1BoundBySuperlinerarGrowth} we conclude
an $\eps$-uniform bound on each addend $\int_{t_{1}}^{t_{2}}\!\!\int_{\Omega}\sum_{i=1}^{2}|J_{i}^{\eps}|\ \d x\ \d t$,
which by summing up implies that $\|\hat{\mu}^{\eps}\|_{TV}$ is $\eps$-uniformly
bounded.
\end{proof}
Next, the compactness result from Lemma \ref{lem:LiminfEstimateAGS}
is used in order to prove compactness for the fluxes and spatial regularity.
\begin{lem}[Regularity for the fluxes and spatial regularity]
\label{lem:RegularityFluxesAndSpaceAndLiminf} Let $(\mu^{\eps})_{\eps>0}$
with $\mu^{\eps}\in\L_{w}^{\infty}([0,T],Q)$ satisfying the a priori
bounds \eqref{eq:AprioriBoundsDissipationFunctional}. Then the corresponding
diffusive fluxes $J^{\eps}:[0,T]\times\Omega\rightarrow\R^{2}$ converge
weakly-star $J^{\eps}\wstarlim J^{0}$ in ${\cal M}([0,T]\times\Omega\times\left\{ 1,2\right\} )$
and $J_{j}^{0}\ll\mu_{j}^{0}$. Moreover $\nabla\rho^{V,\eps}\wstarlim\nabla\rho^{V,0}$
in ${\cal M}([0,T]\times\Omega\times\left\{ 1,2\right\} )$ and $\nabla\rho_{j}^{0}\ll\mu_{j}^{0}$.
In particular, we conclude that $\rho_{j}^{\eps}$ is uniformly bounded
in $\L^{1}([0,T],\W^{1,1}(\Omega))$, which also implies that $\hat{c}^{\eps}$
is uniformly in $\L^{1}([0,T],\W^{1,1}(\Omega))$.
\end{lem}

\begin{proof}
By the bound on the dissipation functional, we get (after extracting
a suitable subsequence of $\eps\rightarrow0$) that $J^{\eps}\wstarlim J^{0}$.
Moreover, we have $\int_{0}^{T}\int_{\Omega}\frac{|J_{j}^{\eps}|^{2}}{\rho_{j}^{\eps}}\d x\leq C$
and $\rho_{j}^{\eps}\wstarlim\rho_{j}^{0}$ . Hence applying the Lemma
\ref{lem:LiminfEstimateAGS}, we conclude that $J_{j}^{0}\ll\mu_{j}^{0}$.
Similarly, we conclude compactness for the gradients $\nabla\rho^{V,\eps}$.
The only thing that remains is to identify the limit. But this is
clear by definition of the weak derivatives, i.e. integrating against
smooth test functions, because this is captured in the weak star convergence.
Lemma \ref{lem:L1BoundBySuperlinerarGrowth} implies that $\rho_{j}^{\eps}$
is uniformly bounded in $\L^{1}([0,T],\W^{1,1}(\Omega))$.
\end{proof}
The spatial regularity and the temporal regularity provides a compactness
result by a BV-generalization of the Aubin-Lions-Simon Lemma.
\begin{thm}[\cite{BarPre86COBS,HePaRe18TMSF}]
\label{thm:AubinLionsForBV}Let $X,Y,Z$ be Banach spaces such that
$X$ is compactly embedded in $Y$, and $Y$ is continuously in $Z^{*}$.
Let $u^{\eps}$ be a bounded sequence in $\L^{1}([0,T],X)$ and in
$BV([0,T],Z^{*})$. Then (up to a sequence) $u^{\eps}$ strongly converges
in $\L^{1}([0,T],Y)$.
\end{thm}

In our situation we immediately conclude that $\hat{c}^{\eps}$ converges
strongly.
\begin{cor}[Strong convergence of coarse-grained variables]
Let $(\mu^{\eps})_{\eps>0}$ with $\mu^{\eps}\in\L_{w}^{\infty}([0,T],Q)$
satisfying the a priori bounds \eqref{eq:AprioriBoundsDissipationFunctional}.
Then the coarse-grained densities $\hat{c^{\eps}}$ converge strongly
in $\L^{1}([0,T]\times\Omega)$.
\end{cor}

\begin{proof}
Lemma \ref{LemmaRegInTime} provides that $\hat{c}^{\eps}$ is bounded
in $BV([0,T],\W^{1,\infty}(\Omega)^{*})$ and Lemma \ref{lem:RegularityFluxesAndSpaceAndLiminf}
provides that $\hat{c}^{\eps}$ is bounded in $\L^{1}([0,T],\W^{1,1}(\Omega))$.
Since the embedding $\W^{1,1}(\Omega)\subset\L^{1}(\Omega)$ is compact
and the embedding $\L^{1}(\Omega)\subset\W^{1,\infty}(\Omega)^{*}$
is continuous, Theorem \ref{thm:AubinLionsForBV} yields that the
sequence $\hat{c^{\eps}}$ is compact in $\L^{1}([0,T]\times\Omega)$.
\end{proof}
It is also clear that we get convergence towards the fast manifold,
which results from the Fisher information of the fast reaction.
\begin{lem}[Convergence towards microscopic equilibrium and strong compactness]
\label{lem:ConvergenceToSlowManifoldAndStongCompactness} Let $(\mu^{\eps})_{\eps>0}$,
$\mu^{\eps}\in\L_{w}^{\infty}([0,T],Q)$ satisfying the a priori bounds
\eqref{eq:AprioriBoundsDissipationFunctional}. Then there is a subsequence
such that $c^{\eps}\rightarrow c^{0}$ strongly in $\L^{1}([0,T]\times\Omega)$
and, moreover, it holds $\rho_{1}^{V,0}=\rho_{2}^{V,0}$ a.e. in $[0,T]\times\Omega$.
\end{lem}

\begin{proof}
The bound on the dissipation functional provides $\int_{0}^{T}\int_{\Omega}\left(\sqrt{\rho_{1}^{V,\eps}}-\sqrt{\rho_{2}^{V,\eps}}\right)^{2}\d x\d t\leq C\eps$.
Hence, we conclude $\|\sqrt{\rho_{1}^{V,\eps}}-\sqrt{\rho_{2}^{V,\eps}}\|_{\L^{2}([0,T]\times\Omega)}\rightarrow0$
as $\eps\rightarrow0$. In particular, we conclude that $\rho_{1}^{V,0}=\rho_{2}^{V,0}$.
The strong convergence towards the slow manifold provides strong convergence
for the whole sequence. Indeed, using Cauchy-Schwartz inequality and
$x-y=\left(\sqrt{x}-\sqrt{y}\right)\left(\sqrt{x}+\sqrt{y}\right)$,
we have
\begin{align*}
\|\rho_{1}^{V,\eps}-\rho_{2}^{V,\eps}\|_{\L^{1}([0,T]\times\Omega)} & \leq\|\sqrt{\rho_{1}^{V,\eps}}-\sqrt{\rho_{2}^{V,\eps}}\|_{\L^{2}([0,T]\times\Omega)}\|\sqrt{\rho_{1}^{V,\eps}}+\sqrt{\rho_{2}^{V,\eps}}\|_{\L^{2}([0,T]\times\Omega)}.
\end{align*}
The last term can be estimated by the AM-GM inequality
\[
\|\sqrt{\rho_{1}^{V,\eps}}+\sqrt{\rho_{2}^{V,\eps}}\|_{\L^{2}([0,T]\times\Omega)}^{2}=\int_{0}^{T}\!\!\!\int_{\Omega}\rho_{1}^{V,\eps}+\rho_{2}^{V,\eps}+2\sqrt{\rho_{1}^{V,\eps}\rho_{2}^{V,\eps}}\d x\d t\leq2\int_{0}^{T}\!\!\!\int_{\Omega}\left(\rho_{1}^{V,\eps}+\rho_{2}^{V,\eps}\right)\d x\d t,
\]
and the right-hand side is bounded since $\mu(t)\in Q$ for $t\in[0,T]$.
Hence, we conclude that $\|\rho_{1}^{V,\eps}-\rho_{2}^{V,\eps}\|_{\L^{1}([0,T]\times\Omega)}\rightarrow0$
as $\eps\to0$.

Using this convergence, we have that also $c_{i}^{\eps}\rightarrow c_{i}^{0}$
strongly in $\L^{1}([0,T]\times\Omega)$. Indeed, we have
\begin{align*}
c_{i}^{\eps}-w_{i}^{V}\frac{c_{1}^{0}+c_{2}^{0}}{w_{1}^{V}+w_{2}^{V}} & =w_{i}^{V}\left(\frac{c_{i}^{\eps}}{w_{i}^{V}}-\frac{c_{1}^{0}+c_{2}^{0}}{w_{1}^{V}+w_{2}^{V}}\right)\\
 & =w_{i}^{V}\left(\frac{c_{i}^{\eps}}{w_{i}^{V}}-\frac{c_{1}^{\eps}+c_{2}^{\eps}}{w_{1}^{V}+w_{2}^{V}}+\frac{c_{1}^{\eps}+c_{2}^{\eps}}{w_{1}^{V}+w_{2}^{V}}-\frac{c_{1}^{0}+c_{2}^{0}}{w_{1}^{V}+w_{2}^{V}}\right)\\
 & =(-1)^{i}\left(\frac{\frac{c_{2}^{\eps}}{w_{2}^{V}}-\frac{c_{1}^{\eps}}{w_{1}^{V}}}{w_{1}^{V}w_{2}^{V}\left(w_{1}^{V}+w_{2}^{V}\right)}\right)+\frac{w_{i}^{V}}{w_{1}^{V}+w_{2}^{V}}\left(c_{1}^{\eps}+c_{2}^{\eps}-(c_{1}^{0}+c_{2}^{0})\right)\ ,
\end{align*}
and both terms converge strongly to zero as $\eps\to0$ by convergence
of $\hat{c}^{\eps}\to\hat{c}^{0}$ and $\rho_{1}^{V,\eps}-\rho_{2}^{V,\eps}\rightarrow0$.
\end{proof}
Finally, we show that the limit $\mu^{0}=c^{0}\d x$ has an absolutely
continuous representative in the space of probability measures. To
do this, we exploit the characterization of absolutely continuous
curves as solutions of the continuity equation following \cite{AmGiSa05GFMS}.
\begin{prop}
\label{prop:AbsolutelyContinuousRepresentative}Let $(\mu^{\eps})_{\eps>0}$,
$\mu^{\eps}\in\L^{\infty}([0,T],Q)$ satisfying the a priori bounds
\eqref{eq:AprioriBoundsDissipationFunctional} and let $c^{0}$ be
the limit of the densities $c^{\eps}$. Then the coarse-grained slow
variable $\hat{\mu}=\mu_{1}^{0}+\mu_{2}^{0}=\left(c_{1}^{0}+c_{2}^{0}\right)\,\d x\in\L_{w}^{\infty}([0,T],\mathrm{Prob}(\Omega))$
has a representative (in time), which is absolutely continuous in
the space of probability measures equipped with the 2-Wasserstein
metric. Moreover, each component $\mu_{i}^{0}$ has an absolutely
continuous representative (in time), which is absolutely continuous
in the space of non-negative Radon measures equipped with the 1-Wasserstein
metric.
\end{prop}

\begin{proof}
The coarse-grained measures $\hat{\mu}^{\eps}$ satisfy continuity
equation $\dot{\hat{\mu}}^{\eps}+\div(\hat{J}^{\eps})=0$ in the sense
of distributions where $\hat{J}^{\eps}=J_{1}^{\eps}+J_{2}^{\eps}$
is the coarse-grained diffusion flux. Since the linear continuity
equation is stable under weak convergence, we conclude that also the
limits satisfy the same continuity equation $\dot{\hat{\mu}}^{0}+\div(\hat{J}^{0})=0$,
where $\hat{J}^{0}$ is the weak{*}-limit of $\hat{J}^{\eps}$ (see
Lemma \ref{lem:RegularityFluxesAndSpaceAndLiminf}). Using \eqref{eq:EstimateCoarseGrainedFluxes},
the bound on the dissipation functional implies a bound $\int_{0}^{T}\int_{\Omega}\mathsf{Q}(\hat{c}^{0},\hat{J}^{0})\d x\d t<\infty$.
Let us define the transport velocity $\hat{v}\in{\cal M}([0,T]\times\Omega,\R)$
\[
\hat{v}=\begin{cases}
\frac{\hat{J}}{\hat{c}} & \mathrm{for}\ \hat{c}>0\\
0 & \mathrm{for}\ \hat{c}=0
\end{cases}\ .
\]
Then $\int_{0}^{T}\int_{\Omega}\mathsf{Q}(\hat{c},\hat{J})\d x\d t=\frac{1}{2}\int_{0}^{T}\int_{\Omega}|\hat{v}|^{2}\hat{c}\d x\d t=\frac{1}{2}\int_{0}^{T}\int_{\Omega}|\hat{v}|^{2}\d\hat{\mu}\d t$
and the bound on the dissipation functional implies the bound on the
Borel velocity field $\|\hat{v}\|_{\L^{2}(\hat{\mu})}<\infty$. Hence,
by Theorem 8.3.1 from \cite{AmGiSa05GFMS} it follows that $t\mapsto\hat{\mu}(t)\in\left(\mathrm{Prob}(\Omega),d_{{\cal W}_{2}}\right)$
has a continuous representative which is absolutely continuous.

To prove time-regularity for $\mu_{i}^{0}$ for $i=1,2$, we first
observe that $\mu_{i}^{0}=\frac{w_{i}^{V}}{w_{1}^{V}+w_{2}^{V}}\hat{\mu}$
is a non-negative Radon measure. To show that it has an absolutely
continuous representative, we proceed as in Lemma \ref{LemmaRegInTime}
and exploit the dual formulation of the 1-Wasserstein distance on
the space of non-negative Radon measures, i.e. integrating against
Lipschitz functions (see e.g. \cite{Edwa11KRT}). Let $\phi\in\mathrm{C}^{1}(\Omega)$
with $\|\phi\|_{\W^{1,\infty}(\Omega)}\leq1$. Using the continuity
equations $\dot{\hat{\mu}}^{0}+\div(\hat{J}^{0})=0$, we conclude
for all $[t_{1},t_{2}]\subset[0,T]$ that 
\begin{align*}
\int_{\Omega}\phi\cdot(\d\mu_{i}^{0}(t_{2})-\d\mu_{i}^{0}(t_{1})) & =\int_{t_{1}}^{t_{2}}\langle\phi,\dot{\mu}_{i}^{0}\rangle\d t=\int_{t_{1}}^{t_{2}}\langle\phi\frac{w_{i}^{V}}{w_{1}^{V}+w_{2}^{V}},\dot{\hat{\mu}}\rangle\d t\\
 & \leq\int_{t_{1}}^{t_{2}}\!\!\int_{\Omega}\nabla\left(\phi\frac{w_{i}^{V}}{w_{1}^{V}+w_{2}^{V}}\right)\cdot\hat{J}\ \d x\ \d t\leq C\int_{t_{1}}^{t_{2}}\!\!\int_{\Omega}|\hat{J}|\ \d x\ \d t,
\end{align*}
where $C=C(w,V)$. The bound on the dissipation functional provides
again that the right-hand side is bounded for each interval $[t_{1},t_{2}]\subset[0,T]$.
Summing up, we conclude that $\mu_{i}^{0}$ has an absolutely continuous
representative, which proves the claim.
\end{proof}

\subsection{Liminf-estimate\label{subsec:Liminf-estimate}}

In this section, we state and prove the liminf-estimate of the $\Gamma$-convergence
result Theorem \ref{thm:GammaConvergenceDissipation}. Once the compactness
is established the proof of the liminf-estimate is comparatively easy.
\begin{thm}
\label{thm:LiminfEstimate} Let $\mu^{\eps}\to\mu^{0}$ in ${\cal \L^{\infty}}_{w}([0,T],Q)$
such that $\sup_{\eps\in]0,1]}\sup_{t\in[0,T]}\cE(\mu^{\eps})<\infty$.
Then, we have the liminf-estimate 
\begin{align*}
\liminf_{\eps\rightarrow0}\fD_{\eps}(\mu^{\eps})\geq\fD_{0}(\mu^{0}),
\end{align*}
where the limit dissipation functional is defined by 
\begin{align*}
\fD_{0}^{V}(\mu)=\begin{cases}
\int_{0}^{T}\cR_{\mathrm{eff}}(\mu,\dot{\mu})+\cR_{\mathrm{eff}}^{*}(\mu,-\D\cE^{V}(\mu))\d t, & \mu\in\mathrm{AC}([0,T],Q),\mu=c\,\d x\ \mathrm{a.e.\ in\ }[0,T]\\
~~\infty & \mathrm{otherwise}
\end{cases}
\end{align*}
with
\begin{align*}
\cR_{\mathrm{eff}}^{*}(\mu,\xi) & =\cR_{\mathrm{diff}}^{*}(\mu,\xi)+\chi_{\left\{ \xi_{1}=\xi_{2}\right\} }(\xi)\\
\cR_{\mathrm{eff}}(\mu,v) & =\inf\left\{ \sum_{j=1}^{2}\int_{\Omega}\widetilde{\mathsf{Q}}(\delta_{j}c_{j},J_{j})\d x:u_{1}+u_{2}=0,\left\{ \begin{array}{c}
v_{1}=-\div J_{1}+u_{1}\\
v_{2}=-\div J_{2}+u_{2}
\end{array}\right\} \right\} .
\end{align*}
\end{thm}

\begin{proof}
We may assume that $\fD_{\eps}(\mu^{\eps})\leq C<\infty$ (otherwise
the claim is trivial). For the given curves $t\mapsto\mu^{\eps}(t)\in Q$
take diffusive fluxes $J^{\eps}$ and reactive fluxes $b^{\eps}$,
which satisfy the generalized continuity equation 
\[
(c,J,b)\in\mathrm{{(gCE)}}\ \ \Leftrightarrow\ \ \left\{ b_{1}+b_{2}=0\ \mathrm{and}\ \left\{ \begin{array}{c}
\dot{c}_{1}=-\div J_{1}+b_{1}\\
\dot{c}_{2}=-\div J_{2}+b_{2}
\end{array}\right\} \right\} ,
\]
and approximate the infimum in $\fD_{\eps}(\mu^{\eps})$ arbitrarily
close, i.e. 
\begin{align*}
\fD_{\eps}(\mu_{\eps})+\eps\geq\int_{0}^{T}\cD_{\eps}(\mu_{\eps},J^{\eps},b^{\eps})\d t.
\end{align*}
The integrand ${\cal D}_{\eps}$ consists of a velocity and a slope
part and both of them split into a reaction and a diffusion part:
\begin{align*}
{\color{black}{\color{black}\cD_{\eps}}}{\color{black}(\mu,J,b)}= & \int_{\Omega}\sum_{j=1}^{2}\widetilde{\mathsf{Q}}\left(\delta_{j}c_{j},J_{j}\right)\d x+\int_{\Omega}\widetilde{\mathsf{C}}\left(\frac{\sqrt{c_{1}c_{2}}}{\eps},b_{2}(x)\right)\d x+\\
 & \quad+\frac{1}{2}\int_{\Omega}\sum_{j=1}^{2}\delta_{j}w_{j}^{V}\frac{\left|\nabla\rho_{j}^{V}\right|^{2}}{\rho_{j}^{V}}\d x+\frac{2}{\eps}\int_{\Omega}\sqrt{w_{1}^{V}w_{2}^{V}}\left(\sqrt{\rho_{1}^{V}}-\sqrt{\rho_{2}^{V}}\right)^{2}\d x\!\\
=: & \mathcal{V}_{\mathrm{diff}}(\mu,J)+\mathcal{V}_{\mathrm{react,}\eps}(\mu,b)+\mathcal{S}_{\mathrm{diff}}(\mu)+\mathcal{S}_{\mathrm{react,}\eps}(\mu)\ .
\end{align*}

Clearly, we also have $\int_{0}^{T}\cD_{\eps}(\mu_{\eps},J^{\eps},b^{\eps})\d t\leq C+\eps<\infty$.
By Lemma \ref{lem:ConvergenceToSlowManifoldAndStongCompactness},
we conclude compactness for the densities $c^{\eps}\rightarrow c^{0}$
in $\L^{1}([0,T]\times\Omega\times\left\{ 1,2\right\} )$ and by Lemma
\ref{lem:RegularityFluxesAndSpaceAndLiminf} that $J^{\eps}\wstarlim J^{0}$
in ${\cal M}([0,T]\times\Omega\times\left\{ 1,2\right\} )$. Using
the lower-semicontinuity result from Lemma \ref{lem:LiminfEstimateAGS}
(which implies the liminf-estimates for ${\cal V}_{\diff}$ and ${\cal S}_{\diff}$)
and that ${\cal V}_{\mathrm{react},\eps},{\cal S}_{\mathrm{react},\eps}\geq0$,
we obtain the estimate
\begin{align*}
\liminf_{\eps\rightarrow0}\int_{0}^{T}\mathcal{D}_{\eps}(\mu^{\eps}) & \d t\geq\int_{0}^{T}\left\{ \int_{\Omega}\sum_{j=1}^{2}\widetilde{\mathsf{Q}}(\delta_{j}c_{j}^{0},J_{j}^{0})\d x+\frac{1}{2}\int_{\Omega}\sum_{j=1}^{2}\delta_{j}w_{j}^{V}\frac{|\nabla\rho_{j}^{0}|^{2}}{\rho_{j}^{0}}\d x\right\} \d t.
\end{align*}

Let us define $A_{i}^{\eps}=\dot{c}_{i}^{\eps}+\div J_{i}^{\eps}$.
We conclude convergence for $A_{i}^{\eps}\rightarrow A_{i}^{0}$ in
the sense of distributions, and, moreover, we have $A_{1}^{\eps}+A_{2}^{\eps}\rightarrow\dot{c}_{1}^{0}+\dot{c}_{2}^{0}+\div J_{1}^{0}+\div J_{2}^{0}=0$.
Let us define $u_{1}:=A_{0}^{2}=\dot{c}_{2}^{0}+\div J_{2}^{0}$ and
$u_{2}:=A_{0}^{1}=\dot{c}_{1}^{0}+\div J_{1}^{0}$. Then $u_{1}+u_{2}=0$,
$A_{1}^{0}+u_{1}=0$ and $A_{2}^{0}+u_{2}=0$. In particular, we conclude
the pointwise estimate
\[
\int_{\Omega}\sum_{j=1}^{2}\widetilde{\mathsf{Q}}(\delta_{j}c_{j}^{0},J_{j}^{0})\d x\geq\inf_{(J,u)}\left\{ \int_{\Omega}\sum_{j=1}^{2}\widetilde{\mathsf{Q}}(\delta_{j}c_{j}^{0},J_{j}^{0})\d x:\left\{ \begin{array}{c}
\dot{c}_{1}+\div J_{1}+u_{1}=0\\
\dot{c}_{2}+\div J_{2}+u_{2}=0\\
u_{1}+u_{2}=0
\end{array}\right\} \right\} \ ,
\]
which finally establishes the liminf-estimate 
\[
\liminf_{\eps\rightarrow0}\int_{0}^{T}\mathcal{D}_{\eps}(\mu^{\eps})\d t\geq\fD_{0}^{V}(\mu).
\]
\end{proof}

\subsection{Construction of the recovery sequence\label{subsec:ConstructionRecoverySequence}}

In this section, we construct the recovery sequence for the functional
$\fD_{0}^{V}$ to finish the $\Gamma$-convergence result in Theorem
\ref{thm:GammaConvergenceDissipation}. To be precise, we will show
the following:
\begin{thm}
\label{thm:LimsupEstimate} Let $\mu^{0}\in\L_{w}^{\infty}([0,T]\times\Omega,\R_{\geq0}^{2})$
such that the a priori bounds $\fD_{0}^{V}(\mu^{0})<\infty$ and $\mathrm{ess\,sup}_{t\in[0,T]}\cE(\mu^{0}(t))<\infty$
hold. Then there is a sequence $(\mu^{\eps})_{\eps>0}$, $\mu^{\eps}\in\mathrm{AC}([0,T],Q)$,
$\sup_{\eps>0}\ \mathrm{ess\,sup}_{t\in[0,T]}\cE(\mu^{\eps}(t))<\infty$,
such that the densities converge $c^{\eps}\to c^{0}$ strongly in
$\L^{1}([0,T]\times\Omega\times\left\{ 1,2\right\} )$ and we have
$\fD_{\eps}^{V}(\mu^{\eps})\rightarrow\fD_{0}^{V}(\mu^{0})$.
\end{thm}

Using Proposition \ref{prop:DissipationFunctionalInCoarseGrainedVariables},
we see that the limit functionals do not contain more information
than the functionals in coarse-grained variables and it holds $\fD_{0}^{V}(\mu^{0})=\hat{\fD}(\hat{\mu}^{0})$.
Hence, we may reconstruct the dissipation functional with the corresponding
diffusion and reaction flux $(J,b)$ from the coarse-grained variables
$(\hat{c},\hat{J})$.

The dissipation functional $\fD_{0}^{V}$ is defined on the space
of general fluctuations around the solution of the evolution equation
\eqref{eq:LRDS}. These fluctuations are neither strictly positive
nor smooth. The proof of the limsup-estimate is done in several steps,
which are elaborated in the next lemmas. The bound $\fD_{0}^{V}(\mu^{0})<\infty$
can be assumed without loss of generality because the other case is
already treated in the liminf-estimate.
\begin{proof}[Proof of Theorem \ref{thm:LimsupEstimate}]
We do the reconstruction in three steps using different approximation
methods. We will do the following steps:
\begin{enumerate}
\item Proposition \ref{prop:ConvOfCTermWithSmoothing} shows that for $\mu^{0}=c^{0}\d x$
with sufficiently smooth and positive density $c^{0}$ the constant
sequence $\mu^{\epsilon,\gamma}=\mu^{0}$ satisfies $\left|\fD_{\eps}(\mu^{\epsilon,\gamma})-\fD_{0}(\mu^{\epsilon,\gamma})\right|\rightarrow0$.
\item Lemma \ref{lem:RecoverySequencePositivity} overcomes the positivity
assumption, i.e. it shows that for all $\mu^{0}=c^{0}\d x$ there
is a positive $c^{\gamma}$ such that $c^{\gamma}\rightarrow c^{0}$
and $\fD_{0}(\mu^{\gamma})\rightarrow\fD_{0}(\mu^{0})$ as $\gamma\to0$.
\item A mollification argument as in \cite[Lemma 8.1.10]{AmGiSa05GFMS}
and stated in Lemma \ref{lem:ConvolutionRecovery} allows us to overcome
regularity by smoothing, which shows $\fD_{0}(\mu^{\epsilon,\gamma})\rightarrow\fD_{0}(\mu^{\gamma})$
as $\epsilon\to0$.
\end{enumerate}
Hence, defining the recovery sequence $\mu^{\eps}:=\mu^{\epsilon,\gamma}$,
we have
\[
\left|\fD_{\eps}(\mu^{\eps})-\fD_{0}(\mu^{0})\right|\leq\left|\fD_{\eps}(\mu^{\epsilon,\gamma})-\fD_{0}(\mu^{\epsilon,\gamma})\right|+\left|\fD_{0}(\mu^{\epsilon,\gamma})-\fD_{0}(\mu^{\gamma})\right|+\left|\fD_{0}(\mu^{\gamma})-\fD_{0}(\mu^{0})\right|,
\]
where the first term tends to zero by the first reconstruction step,
the second term tends to zero by the third reconstruction step and
the third term tends to zero by the second reconstruction step, which
in total proves the desired convergence.
\end{proof}
Before performing the three recovery steps in Section \ref{subsec:Auxiliary-results-for-Recovery-Sequence},
we first illustrate the general idea of constructing the recovery
sequence by forgetting about positivity and regularity issues for
the first moment.

\subsubsection{Construction of recovery sequence for smooth and positive measures}

To show the general idea, let us firstly assume that the density of
$\hat{\mu}$ is sufficiently smooth and positive, i.e. we assume that
its Lebesgue density satisfies $\hat{c}\geq\frac{1}{C}>0$ on $\Omega\times[0,T]$
and has a enough regularity that will be specified below. Let $\hat{J}$
be the diffusion flux which provides the minimum in $\hat{\fD}(\hat{\mu}^{0})=\fD_{0}^{V}(\mu)$
and satisfies $\dot{\hat{c}}+\mathrm{div}(\hat{J})=0$. We define
the reconstructed variables by 
\begin{align}
c_{1} & =\frac{w_{1}^{V}}{w_{1}^{V}+w_{2}^{V}}\hat{c},\quad c_{2}=\frac{w_{2}^{V}}{w_{1}^{V}+w_{2}^{V}}\hat{c},\quad J_{1}=\frac{\delta_{1}w_{1}^{V}}{\delta_{1}w_{1}^{V}+\delta_{2}w_{2}^{V}}\hat{J},\quad J_{2}=\frac{\delta_{2}w_{2}^{V}}{\delta_{1}w_{1}^{V}+\delta_{2}w_{2}^{V}}\hat{J}.\nonumber \\
b_{1} & =\left(\frac{\delta_{1}-\delta_{2}}{\delta_{1}w_{1}^{V}+\delta_{2}w_{2}^{V}}\frac{w_{1}^{V}w_{2}^{V}}{w_{1}^{V}+w_{2}^{V}}\right)\div\hat{J}+\hat{J}\cdot\nabla\left(\frac{\delta_{1}w_{1}^{V}}{\delta_{1}w_{1}^{V}+\delta_{2}w_{2}^{V}}\right),\quad b_{2}=-b_{1}.\label{eq:Reconstruction}
\end{align}

The reconstructed concentrations $c$ and diffusion fluxes $J=(J_{1},J_{2})$
are proportional to the coarse grained concentration $\hat{c}$ and
diffusion flux $\hat{J}$, respectively. On the coarse-grained level,
which considers only one species there is no reaction flux anymore.
(This changes when considering large reaction-diffusion system as
explained in Section \ref{sec:Multispecies}). The reactive flux $b=(b_{1},b_{2})$
is given as a function of the coarse-grained diffusion flux $\hat{J}$,
which means that in the limit the diffusion determines the hidden
reaction.

Concerning regularity issues, we immediately observe the following.
Since $w^{V}$ is smooth and positive, $c_{1},c_{2}$ have the same
regularity as $\hat{c}$ and also $J_{1},J_{2}$ have the same regularity
as $\hat{J}$. Only the reaction fluxes $b_{i}$ are a priori not
well-defined (e.g. in $\L^{1}$) if $\div\hat{J}$ and $\hat{J}$
are not regular enough. This means, that the reaction flux between
the fast-connected species is not well-defined for general $\hat{J}\in{\cal M}(\Omega\times[0,T],\R^{3})$.
Note, that regularity assumptions for $\div\hat{J}$ are not needed
if $\delta_{1}=\delta_{2}$, i.e. if both species diffuse with the
same diffusion constant. In particular, in this situation no regularization
argument as in Lemma \ref{lem:ConvolutionRecovery} is necessary.
Moreover, no additional regularity for $\hat{J}$ is needed if $\frac{\delta_{1}w_{1}^{V}}{\delta_{1}w_{1}^{V}+\delta_{2}w_{2}^{V}}=\theta\in\ ]0,1[$
is constant. This is equivalent to 
\[
V_{1}(x)-V_{2}(x)=\log\left(\frac{1-\theta}{\theta}\frac{\delta_{1}}{\delta_{2}}\frac{\beta}{\alpha}\right)=\mathrm{const},
\]
which means that the potentials $V_{1},V_{2}$ differ in a constant
on $\Omega$. In particular, this implies that for the coarse-grained
potential $\hat{V}$ we have $\nabla\hat{V}=\nabla V_{1}=\nabla V_{2}$.
As we will see, enough regularity for $\hat{J}$ is already obtained
from bounds on the dissipation functional. Of course, regularity properties
for $\div\hat{J}$ and $\hat{J}$ are independent of each other.

So let us assume for the moment that $b_{i}$ is well-defined. Then
we conclude $(c,J,b)\in\mathrm{(gCE)}$, because we have
\begin{align*}
\dot{c}_{1}+\div J_{1} & =\frac{w_{1}^{V}}{w_{1}^{V}+w_{2}^{V}}\dot{\hat{c}}+\div\left(\frac{\delta_{1}w_{1}^{V}}{\delta_{1}w_{1}^{V}+\delta_{2}w_{2}^{V}}\hat{J}\right)=\\
 & =\frac{w_{1}^{V}}{w_{1}^{V}+w_{2}^{V}}\dot{\hat{c}}+\hat{J}\cdot\nabla\left(\frac{\delta_{1}w_{1}^{V}}{\delta_{1}w_{1}^{V}+\delta_{2}w_{2}^{V}}\right)+\left(\frac{\delta_{1}w_{1}^{V}}{\delta_{1}w_{1}^{V}+\delta_{2}w_{2}^{V}}\right)\div\hat{J}\\
 & =-\frac{w_{1}^{V}}{w_{1}^{V}+w_{2}^{V}}\div\hat{J}+\hat{J}\cdot\nabla\left(\frac{\delta_{1}w_{1}^{V}}{\delta_{1}w_{1}^{V}+\delta_{2}w_{2}^{V}}\right)+\left(\frac{\delta_{1}w_{1}^{V}}{\delta_{1}w_{1}^{V}+\delta_{2}w_{2}^{V}}\right)\div\hat{J}=b_{1},
\end{align*}
where we used that $(\hat{c},\hat{J})$ solves $\dot{\hat{c}}+\mathrm{div}\hat{J}=0$.
Similarly, we see that $\dot{c}_{2}+\div J_{2}=b_{2}$ and, by definition,
we have $b_{1}+b_{2}=0$. Moreover, boundary properties of $\hat{J}$
remain for $J=(J_{1},J_{2})$.

Since $\frac{c_{1}}{w_{1}^{V}}=\frac{c_{2}}{w_{2}^{V}}$, we conclude
that $\fD_{\eps}^{V}(\mu)\leq\fD_{0}(\mu)+\int_{0}^{T}\int_{\Omega}\widetilde{\mathsf{C}}\left(\frac{\sqrt{c_{1}c_{2}}}{\eps},b_{2}\right)\d x\d t$.
That means, that for proving that the constant sequence $\mu^{\eps}=\mu$
is a recovery sequence, it suffice to show that $\int_{0}^{T}\mathsf{\int_{\Omega}}\widetilde{\mathsf{C}}\left(\frac{\sqrt{c_{1}c_{2}}}{\eps},b_{2}\right)\d x\d t\rightarrow0$
as $\eps\to0$. This is, in fact, shown in the next lemma under the
assumption that $\div\hat{J},\hat{J}\in\L^{\mathsf{C}}([0,T]\times\Omega)$.
The proof basically uses the monotonicity property of the Legendre
dual function $\widetilde{\mathsf{C}}(a_{1},b)\leq\widetilde{\mathsf{C}}(a_{2},b)$
as $a_{1}\geq a_{2}$ (see Lemma \ref{lem:Convexity}), its superlinear
growth and the dominated convergence theorem.
\begin{lem}
\label{lem:BinLCimpliesCepstozero}Let $\hat{c}\in\L^{1}(\Omega\times[0,T])$
with $\hat{c}\geq\frac{1}{C}$ a.e. in $[0,T]\times\Omega$ for a
constant $C>0$, and let $\hat{J}:\Omega\to\R^{d}$ satisfy $\div\hat{J},|\hat{J}|\in\L^{\mathsf{C}}([0,T]\times\Omega)$.
Then $\int_{0}^{T}\!\!\mathsf{\int_{\Omega}}\widetilde{\mathsf{C}}\left(\frac{\sqrt{c_{1}c_{2}}}{\eps},b_{2}\right)\d x\d t\rightarrow0$
as $\eps\to0$.
\end{lem}

\begin{proof}
Since $\hat{c}\geq\frac{1}{C}$ a.e., Lemma \ref{lem:Convexity} yields
$\widetilde{\mathsf{C}}\left(\frac{\sqrt{c_{1}c_{2}}}{\eps},b_{2}\right)\leq\widetilde{\mathsf{C}}\left(\frac{1}{C\eps},b_{2}\right)=\frac{1}{C\eps}\mathsf{C}\left(C\eps b_{2}\right)$.
Moreover, we have the estimate
\[
\tfrac{1}{2}|r|\log(|r|+1)\leq\mathsf{C}(r)\leq2|r|\log(|r|+1),
\]
which implies
\[
\int_{0}^{T}\mathsf{\int_{\Omega}}\widetilde{\mathsf{C}}\left(\frac{\sqrt{c_{1}c_{2}}}{\eps},b_{2}\right)\d x\d t\leq\tfrac{1}{C\eps}\int_{0}^{T}\int_{\Omega}\mathsf{C}\left(C\eps b_{2}\right)\d x\d t\leq\int_{0}^{T}\int_{\Omega}2|b_{2}|\log(C\eps|b_{2}|+1)\d x\d t.
\]
By assumption, we have that $\div\hat{J},|\hat{J}|\in\L^{\mathsf{C}}([0,T]\times\Omega)$.
Since $V\in\mathrm{C}^{1}(\overline{\Omega})$ and the Orlicz space
$\L^{\mathsf{C}}([0,T]\times\Omega)$ is a Banach space, we conclude
that $b_{2}\in\L^{\mathsf{C}}([0,T]\times\Omega)$. By the inequality
for $\mathsf{C}$, this implies that for $\eps<\frac{1}{C}$, the
right-hand side is bounded. We show that (for a subsequence) the integrand
converges to zero pointwise a.e. in $[0,T]\times\Omega$. By the dominated
convergence theorem, this would imply that $\int_{0}^{T}\mathsf{\int_{\Omega}}\widetilde{\mathsf{C}}\left(\frac{\sqrt{c_{1}c_{2}}}{\eps},b_{2}\right)\d x\d t\rightarrow0$
as $\eps\to0$.

To see that the integrand converges to zero pointwise, we firstly
observe that $b_{2}\in\L^{\mathsf{C}}([0,T]\times\Omega)\subset\L^{1}([0,T]\times\Omega)$,
which means that
\[
\int_{0}^{T}\int_{\Omega}\log(\eps|b_{2}|+1)\d x\d t\leq\eps\int_{0}^{T}\int_{\Omega}|b_{2}|\d x\d t\rightarrow0.
\]
Hence, (for a subsequence) $\log(\eps|b_{2}|+1)$ converges pointwise
to zero and, thus, also $|b_{2}|\log(C\eps|b_{2}|+1)$.
\end{proof}
In fact, the above proof is quite robust and already suggests that
the same convergence holds even if $\hat{c}^{\eps}\to0$ not to fast
somewhere in $\Omega\times[0,T]$ and $\|b_{\eps}^{2}\|\approx\eps^{-\alpha}$
for some $\alpha>0$. This is proved in Proposition \ref{prop:ConvOfCTermWithSmoothing}.

In the following, we need to overcome the positivity assumption $\hat{c}\geq\frac{1}{C}$
and the regularity assumption for $\hat{J}$ and $\div\hat{J}$. The
first is done by shifting the density $\hat{c}^{\delta}:=\frac{1}{Z_{\delta}}(\hat{c}+\delta)$,
$\delta>0$; the necessary regularity of $\hat{J}$ is provided immediately
by the bound on the dissipation functional; the regularity of $\div\hat{J}$
is achieved by smoothing using that $(\hat{c},\hat{J})$ is a solution
of the coarse-grained continuity equation $\dot{\hat{c}}+\div\hat{J}=0$.

\subsubsection{Auxiliary results for construction of recovery sequence for general
measures\label{subsec:Auxiliary-results-for-Recovery-Sequence}}

First, we show how to overcome the positivity assumption. This is
done by a controlled positive shift.
\begin{lem}
\label{lem:RecoverySequencePositivity}For all $\mu^{0}=c^{0}\,\d x$
satisfying $\fD_{0}^{V}(\mu^{0})=\hat{\fD}(\hat{\mu}^{0})<\infty$
there is a sequence $(\hat{c}^{\gamma})$ of densities satisfying
$\hat{c}^{\gamma}\geq\gamma$, $\hat{c}^{\gamma}\to\hat{c}$ in $\L^{1}([0,T]\times\Omega)$
such that for their corresponding measures we have $\hat{\fD}(\hat{\mu}^{\gamma})=\fD_{0}(\mu^{\gamma})\rightarrow\fD_{0}(\mu^{0})=\hat{\fD}(\mu^{0})$
as $\gamma\to0$ and $\sup_{\gamma\in]0,1]}\mathrm{ess\,sup}_{t\in[0,T]}\cE(\mu^{\gamma}(t))<\infty$.
\end{lem}

\begin{proof}
For small $\gamma>0$, we define $\hat{c}^{\gamma}:=\frac{1}{Z_{\gamma}}(\hat{c}+2\gamma)$
, where $Z_{\gamma}=1+2\gamma|\Omega|=1+2\gamma>0$ is the normalization
factor such that $\int_{\Omega}\hat{c}^{\gamma}\d x=1$. Hence, $Z_{\gamma}\searrow1$,
$\hat{c}^{\gamma}\rightarrow\hat{c}$, $\sup_{\gamma\in]0,1]}\mathrm{ess\,sup}_{t\in[0,T]}\cE(\mu^{\gamma}(t))<\infty$
and w.l.o.g. we assume that $\hat{c}^{\gamma}\geq\gamma$. Moreover,
we define $\hat{J}^{\gamma}:=\frac{1}{Z_{\gamma}}\hat{J}$. Clearly,
$\hat{J}^{\gamma}\cdot\nu=0$ on $\partial\Omega$ and $(\hat{c}^{\gamma},\hat{J}^{\gamma})$
solves the continuity equation $\dot{\hat{c}}^{\gamma}+\div\hat{J}^{\gamma}=0$.
We compute the terms in the dissipation functional $\fD_{0}^{V}(\mu^{\gamma}).$
We have $\hat{\rho}^{V,\gamma}=\frac{\hat{c}^{\gamma}}{w^{V}}=\frac{1}{Z_{\gamma}}\frac{\hat{c}+2\gamma}{w^{V}}$.
Using the bounds $\sup_{x\in\overline{\Omega}}\left\{ \nabla\left(1/w^{V}\right),w^{V}\right\} \leq C$,
we get
\begin{align*}
 & \nabla\hat{\rho}^{V,\gamma}=\frac{1}{Z_{\gamma}}\left\{ \nabla\left(\frac{\hat{c}}{w^{V}}\right)+2\gamma\nabla\left(\frac{1}{w^{V}}\right)\right\} \ \ \Rightarrow\ \ \left|\nabla\hat{\rho}^{V,\gamma}\right|\leq\frac{1}{Z_{\gamma}}\left\{ \left|\nabla\hat{\rho}^{V}\right|+2\gamma C\right\} 
\end{align*}
Using the estimates $\frac{1}{a+\delta}\leq\frac{1}{a}$, $\frac{1}{Z_{\gamma}}\leq1$
and the inequality $2xy\leq\sqrt{\gamma}x^{2}+\frac{1}{\sqrt{\gamma}}y^{2}$
in the second estimate, we get the pointwise estimate
\begin{align*}
\frac{|\nabla\hat{\rho}^{V,\gamma}|^{2}}{\hat{\rho}^{V,\gamma}} & =\frac{1}{Z_{\gamma}}\frac{\left(\left|\nabla\hat{\rho}^{V}\right|+2\gamma C\right)^{2}}{\hat{\rho}^{V}+\frac{2\gamma}{w^{V}}}=\frac{1}{Z_{\gamma}}\frac{\left|\nabla\hat{\rho}^{V}\right|^{2}+4\gamma C\left|\nabla\hat{\rho}^{V}\right|+4\gamma^{2}C^{2}}{\hat{\rho}^{V}+\frac{2\gamma}{w^{V}}}\\
 & \leq\frac{\left|\nabla\hat{\rho}^{V}\right|^{2}}{\hat{\rho}^{V}}+2\frac{2\gamma C\left|\nabla\hat{\rho}^{V}\right|}{\hat{\rho}^{V}+\frac{2\gamma}{w^{V}}}+\frac{4\gamma^{2}C^{2}}{\hat{\rho}^{V}+\frac{2\gamma}{w^{V}}}\\
 & \leq\frac{\left|\nabla\hat{\rho}^{V}\right|^{2}}{\hat{\rho}^{V}}+\frac{1}{\hat{\rho}^{V}+\frac{2\gamma}{w^{V}}}\left(\sqrt{\gamma}\left|\nabla\hat{\rho}^{V}\right|^{2}+\tfrac{1}{\sqrt{\gamma}}\left\{ 2\gamma C\right\} ^{2}\right)+2\gamma C^{2}w^{V}\\
 & \leq\frac{\left|\nabla\hat{\rho}^{V}\right|^{2}}{\hat{\rho}^{V}}(1+\sqrt{\gamma})+2\sqrt{\gamma}w^{V}C^{2}(1+\sqrt{\gamma}).
\end{align*}
Hence, $\int_{0}^{T}\int_{\Omega}\delta^{V}w^{V}\frac{|\nabla\hat{\rho}^{V,\gamma}|^{2}}{\hat{\rho}^{V,\gamma}}\d x\d t\rightarrow\int_{0}^{T}\int_{\Omega}\delta^{V}w^{V}\frac{|\nabla\hat{\rho}^{V}|^{2}}{\hat{\rho}^{V}}\d x\d t$
as $\gamma\to0$.

Similarly, we get 
\[
\int_{0}^{T}\int_{\Omega}\frac{|\hat{J}^{\gamma}|^{2}}{\hat{\delta}\hat{c}^{\gamma}}\d x\d t=\frac{1}{Z_{\gamma}}\int_{0}^{T}\int_{\Omega}\frac{|\hat{J}|^{2}}{\hat{\delta}\left(\hat{c}+2\gamma\right)}\d x\d t\leq\int_{0}^{T}\int_{\Omega}\frac{|\hat{J}|^{2}}{\hat{\delta}\hat{c}}\d x\d t,
\]
which implies that $\int_{0}^{T}\!\!\int_{\Omega}\widetilde{\mathsf{Q}}(\hat{\delta}\hat{c}^{\gamma},\hat{J}^{\gamma})\d x\d t\leq\int_{0}^{T}\!\!\int_{\Omega}\widetilde{\mathsf{Q}}(\hat{\delta}\hat{c},\hat{J})\d x\d t$.
Hence, we conclude that $\hat{c}^{\gamma}\geq\gamma$, $\hat{c}^{\gamma}\to\hat{c}$
and $\hat{\fD}(\hat{c}^{\gamma})\rightarrow\hat{\fD}(\hat{c})$ as
$\gamma\to0$.
\end{proof}
Next, we are going to show that the flux $b_{1}=-b_{2}$ can be made
sufficiently smooth, i.e. at least in $\L^{\mathsf{C}}$ which would
allow us to proceed similar as in Lemma \ref{lem:BinLCimpliesCepstozero}.

Recalling the formula for the reconstructed flux, we have 
\[
b_{2}=\left(\frac{\delta_{1}-\delta_{2}}{\delta_{1}w_{1}^{V}+\delta_{2}w_{2}^{V}}\frac{w_{1}^{V}w_{2}^{V}}{w_{1}^{V}+w_{2}^{V}}\right)\div\hat{J}+\hat{J}\cdot\nabla\left(\frac{\delta_{1}w_{1}^{V}}{\delta_{1}w_{1}^{V}+\delta_{2}w_{2}^{V}}\right)=:a_{1}\div\hat{J}+\hat{J}\cdot a_{2}
\]
and $b_{2}=-b_{1}$, where $a_{1}\in\mathrm{C}^{1}(\overline{\Omega},\R)$,
$a_{2}\in\mathrm{C}^{0}(\overline{\Omega},\R^{d})$ such that $\sup_{x\in\overline{\Omega}}\left\{ |a_{1}(x)|,|a_{2}(x)|\right\} \leq C$.
In particular, the regularity of $b_{1}=-b_{2}$ does not depend on
$a_{1},a_{2}$. We are going to prove that $\div\hat{J}$ and $\hat{J}$
have enough regularity. The regularity of $\hat{J}$ follows from
the bound on the dissipation functional. The regularity of $\div\hat{J}$
is achieved by mollification.

First, we show that $\hat{J}\in\L^{\tilde{p}}([0,T]\times\Omega),$
for some $\tilde{p}>1$. Clearly, we have $\hat{J}\in\L^{1}([0,T]\times\Omega)$
by Lemma \ref{lem:L1BoundBySuperlinerarGrowth} and the bound on the
dissipation functional. To improve the regularity of $\hat{J}$, we
firstly improve the regularity of $\hat{c}\in\L^{p}([0,T]\times\Omega)$
which is provided from the bound on the dissipation functional $\fD_{0}$
and the energy functional $\cE$. The bound on the energy yields $\hat{c}\in\L^{\infty}([0,T],\L^{1}(\Omega))$.
The bound on the Fisher-information yields again by Lemma \ref{lem:L1BoundBySuperlinerarGrowth}
that $\hat{c}\in\L^{1}([0,T],\W^{1,1}(\Omega))$ as in Lemma \ref{lem:RegularityFluxesAndSpaceAndLiminf}.
By the Sobolev embedding theorem, we have the compact embedding $\W^{1,1}(\Omega)\subset\L^{q}(\Omega)$,
where $1-\tfrac{1}{d}>\tfrac{1}{q}\Leftrightarrow q<\tfrac{d}{d-1}$.
Thus $\hat{c}\in\L^{\infty}([0,T],\L^{1}(\Omega))\cap\L^{1}([0,T],\L^{\tfrac{d}{d-1}}(\Omega))$.
Using the next classical interpolation result we get $\hat{c}\in\L^{p}([0,T]\times\Omega)$
for some $p>1$.
\begin{thm}[Theorem 5.1.2, \cite{BerLoe76ISAI}]
 \label{thm:InterpolationTheorem}Let $X_{1},X_{2}$ be Banach spaces.
Then for the complex interpolation spaces, it holds for any $\theta\in]0,1[$
\[
[\L^{p_{1}}([0,T],X_{1}),\L^{p_{2}}([0,T],X_{2})]_{\theta}\simeq\L^{p_{\theta}}([0,T],[X,Y]_{\theta}),
\]
where $\tfrac{1}{p_{\theta}}=\tfrac{1-\theta}{p_{1}}+\tfrac{\theta}{p_{2}}$.
\end{thm}

In our particular situation, we have the following.
\begin{lem}
\label{lem:IntegrabilityConcentrationc}Let $\mu^{0}\in\L_{w}^{\infty}([0,T],Q)$
such that $\fD_{0}^{V}(\mu^{0})<\infty$ and $\mathrm{ess\,sup}_{t\in[0,T]}\cE(\mu^{0}(t))<\infty$.
Then the density $\hat{c}^{0}$ is in $\L^{p}([0,T]\times\Omega)$
with $p=\frac{d+1}{d}>1$.
\end{lem}

\begin{proof}
To apply Theorem \ref{thm:InterpolationTheorem}, we first observe
that the Lebesgue spaces form interpolation couples. In our situation
we have $p_{1}=1,p_{2}=\infty$, $X_{1}=\L^{\tfrac{d}{d-1}}(\Omega),X_{2}=\L^{1}(\Omega)$.
Hence, $p_{\theta}=\tfrac{1}{1-\theta}>1$. Moreover, $[X,Y]_{\theta}=[\L^{q}(\Omega),\L^{1}(\Omega)]_{\theta}\simeq\L^{q_{\theta}}(\Omega)$,
where $\tfrac{1}{q_{\theta}}=\tfrac{1-\theta}{q}+\tfrac{\theta}{1}$.
Setting $p_{\theta}=q_{\theta}$, we conclude $\tfrac{1-\theta}{1-2\theta}=q=\tfrac{d}{d-1}$.
Solving $p_{\theta}=q_{\theta}$ for $\theta$, we obtain $\theta=\tfrac{1}{1+d}$,
and hence $p_{\theta}=q_{\theta}=\tfrac{d+1}{d}>1$. Summarizing,
we conclude 
\begin{equation}
\hat{c}\in\,\L^{\infty}([0,T],\L^{1}(\Omega))\ \cap\ \L^{1}([0,T],\W^{1,1}(\Omega))\ \subset\ \L^{\tfrac{d+1}{d}}([0,T]\times\Omega).\label{eq:RegForC}
\end{equation}
\end{proof}
\begin{rem}
In particular, if $d=2$, then $\hat{c}\in\L^{3/2}([0,T]\times\Omega)$
and if $d=3$, then $\hat{c}\in\L^{4/3}([0,T]\times\Omega)$. Iterating
the procedure, it is even possible to obtain $\hat{c}\in\L^{(d+2)/d}(([0,T]\times\Omega))$
for $d\geq2.$
\end{rem}

Knowing integrability of $\hat{c}$, we get also better integrability
of the fluxes $\hat{J}\in\L^{\tilde{p}}([0,T]\times\Omega)$ for some
$\tilde{p}>1$, which follows from the next lemma.
\begin{lem}
\label{lem:IntegrabilityFluxJ}$\ $
\begin{enumerate}
\item We have for all $J\in\R^{d}$ and $c>0$ that 
\[
\frac{|J|^{2}}{c}+\frac{1}{p}c^{p}\geq\left(1+\frac{1}{p}\right)|J|^{\frac{2p}{p+1}}.
\]
\item Let $\mu^{0}\in\L_{w}^{\infty}([0,T]\times\Omega,\R_{\geq0}^{2})$
such that $\fD_{0}^{V}(\mu^{0})<\infty$ and $\mathrm{ess\,sup}_{t\in[0,T]}\cE(\mu^{0}(t))<\infty$
and let $\hat{J}\in\mathcal{M}([0,T]\times\Omega,\R^{d})$ be the
corresponding diffusion flux satisfying the continuity equation. Then
$\hat{J}\in\L^{\tilde{p}}([0,T]\times\Omega,\R^{d})$ for $\tilde{p}=\frac{2d+2}{2d+1}>1$.
\end{enumerate}
\end{lem}

\begin{proof}
For proving the first part, let us define for fixed $J\in\R^{d}$
the function $F:]0,\infty[\rightarrow\R,\ F(c):=\frac{|J|^{2}}{c}+\frac{1}{p}c^{p}$.
Clearly, $F\geq0$ and $F(c)\rightarrow\infty$ as $c\rightarrow0$
or $c\rightarrow\infty$. We compute the minimum. We have $F'(c)=-|J|^{2}c^{-2}+c^{p-1}$
and hence the critical point is at $c_{0}=|J|^{2/(p+1)}$. Inserting
$c_{0}$ into $F$ we get $F(c)\geq F(c_{0})=|J|^{2}|J|^{-2/(p+1)}+\frac{1}{p}|J|^{2p/(p+1)}=(1+\frac{1}{p})|J|^{2p/(p+1)}$,
which proves the claim.

For second part, we use that by Lemma \ref{lem:IntegrabilityConcentrationc}
we have $\hat{c}\in\L^{p}([0,T]\times\Omega)$ for $p=\frac{d+1}{d}$.
This implies by the first part that $\hat{J}\in\L^{\tilde{p}}([0,T]\times\Omega,\R^{d})$
for $\tilde{p}=\frac{2\frac{d+1}{d}}{\frac{d+1}{d}+1}=\frac{2d+2}{2d+1}$.
\end{proof}
To obtain regularity for the whole reaction flux $b_{1}=-b_{2}$,
we have to get regularity also for $\div\hat{J}$. This is done by
mollifying the solution $(\hat{c},\hat{J})$ of the continuity equation
$\dot{\hat{c}}+\div\hat{J}=0$ in time. We already now that $\hat{c}:[0,T]\to\mathrm{Prob}(\Omega)$
is continuous by Lemma \ref{prop:AbsolutelyContinuousRepresentative}.
With a slight abuse of notation, we denote by $\hat{c}$ also a continuous
continuation on $\R$ such that $\hat{c}\in\L^{p}(\R,\L^{p}(\Omega))$.
Now, we mollify in time and define $\hat{c}^{\epsilon}(t)=\int_{\R}\hat{c}(s)\psi_{\epsilon}(t-s)\d s$
where $\psi_{\epsilon}$ is a positive and symmetric mollifier. Analogously,
we define $\hat{J}^{\epsilon}$ by convolution, i.e. $\hat{J}^{\epsilon}(t)=\int_{\R}\hat{J}(s)\psi_{\epsilon}(t-s)\d s$.
Since the continuity equation is linear, the smoothed functions $(\hat{c}^{\epsilon},\hat{J}^{\epsilon})$
satisfy again the continuity equation with the same no-flux boundary
conditions.

The next lemma shows that the dissipation functional can be approximated
by mollifying $(\hat{c},\hat{J})$. This basically uses the convexity
of $\hat{\fD}$ in $(\hat{c},\hat{J})$.
\begin{lem}
\label{lem:ConvolutionRecovery} Let $\mu^{0}\in\L_{w}^{\infty}([0,T]\times\Omega,\R_{\geq0}^{2})$
such that the a priori bounds $\fD_{0}^{V}(\mu^{0})<\infty$ and $\mathrm{ess\,sup}_{t\in[0,T]}\cE(\mu^{0}(t))<\infty$
hold. Let $\hat{J}\in\mathcal{M}([0,T]\times\Omega,\R^{d})$ be the
corresponding diffusion flux satisfying the continuity equation. Let
$\psi^{\epsilon}:\R\rightarrow\R$ be a positive and symmetric mollifier.
Define $\hat{c}^{\epsilon}(t)=\int_{\R}\hat{c}(s)\psi_{\epsilon}(t-s)\d s$
and $\hat{J}^{\epsilon}(t)=\int_{\R}\hat{J}(s)\psi_{\epsilon}(t-s)\d s$.
Then, we have $\sup_{\epsilon\in]0,1]}\mathrm{ess\,sup}_{t\in[0,T]}\cE(\mu^{\eps}(t))<\infty$
and 
\begin{align}
 & \int_{0}^{T}\left\{ \int_{\Omega}\sum_{j=1}^{2}\widetilde{\mathsf{Q}}(\delta_{j}c_{j}^{\epsilon},J_{j}^{\epsilon})\d x+\frac{1}{2}\int_{\Omega}\sum_{j=1}^{2}\delta_{j}w_{j}^{V}\frac{|\nabla\rho_{j}^{\epsilon}|^{2}}{\rho_{j}^{\epsilon}}\d x\right\} \d t\label{eq:ConvergenceDissipationFunctional}\\
 & \quad\rightarrow\int_{0}^{T}\left\{ \int_{\Omega}\sum_{j=1}^{2}\widetilde{\mathsf{Q}}(\delta_{j}c_{j}^{0},J_{j}^{0})\d x+\frac{1}{2}\int_{\Omega}\sum_{j=1}^{2}\delta_{j}w_{j}^{V}\frac{|\nabla\rho_{j}^{0}|^{2}}{\rho_{j}^{0}}\d x\right\} \d t,\nonumber 
\end{align}
which in particular, implies that $\fD_{\eps}^{V}(\mu^{\eps})\to\fD_{0}^{V}(\mu^{0})$.
\end{lem}

\begin{proof}
The energy bound on $\mu^{\epsilon}$ is trivially satisfied. The
convergence for the dissipation functional \eqref{eq:ConvergenceDissipationFunctional}
follows directly as the proof of Lemma 8.1.10 in \cite{AmGiSa05GFMS}
since the integrand is convex in $(\hat{c},\hat{J})$.
\end{proof}
With these preparations, we are able to show the remaining step in
the proof of Theorem \ref{thm:LimsupEstimate}. The next proposition
shows in analogy to Lemma \ref{lem:BinLCimpliesCepstozero} that the
contribution by the reaction flux $b_{2}^{\epsilon}$ to the dissipation
functional $\fD_{\eps}^{V}$ converges to zero as $\epsilon\to0$.
\begin{prop}
\label{prop:ConvOfCTermWithSmoothing} Let $\mu^{0}\in\L_{w}^{\infty}([0,T]\times\Omega,\R_{\geq0}^{2})$
such that $\fD_{0}^{V}(\mu^{0})=\hat{\fD}_{0}(\hat{\mu})<\infty$
and $\mathrm{ess\,sup}_{t\in[0,T]}\cE(\mu^{0}(t))<\infty$ and let
$\hat{J}\in\mathcal{M}([0,T]\times\Omega,\R^{d})$ be the corresponding
diffusion flux satisfying the continuity equation. Let $\psi^{\epsilon}:\R\to\R$
be a positive and symmetric mollifier, which is specified below. Let
$\hat{c}^{\epsilon},\hat{J}^{\epsilon}$ the mollified functions as
in Lemma \ref{lem:ConvolutionRecovery} and $\hat{c}^{\epsilon,\gamma},\hat{J}^{\epsilon,\gamma}$
the mollified and shifted functions as in Lemma \ref{lem:RecoverySequencePositivity}.

Let $\psi^{\eps}$ be such that $\|\dot{\hat{c}}^{\epsilon}\|_{\L^{\tilde{p}}([0,T]\times\Omega)}\lesssim\frac{1}{\epsilon^{\alpha}}$
and $\gamma\geq C\epsilon^{1-\lambda}$ for $\epsilon\to0$ , where
$\tilde{p}$ is the integrability exponent of the fluxes as in Lemma
\ref{lem:IntegrabilityFluxJ}, $C>0$ is a positive constant and $\lambda\in[0,1[$,
$\alpha\in[0,1]$ satisfies the inequality $d+1\leq\frac{\lambda}{2\alpha}$.
Then, we have $|\fD_{\eps}(\mu^{\epsilon,\gamma})-\fD_{0}(\mu^{\epsilon,\gamma})|\to0$.
\end{prop}

\begin{proof}
First, we observe that for given $\hat{c}$ such a mollifier and these
constants $\alpha,\lambda$ satisfying all the conditions can be easily
constructed.

To prove the convergence, we follow the same strategy as in the proof
of Lemma \ref{lem:BinLCimpliesCepstozero}. Defining the reconstructed
concentrations and fluxes as in \eqref{eq:Reconstruction}, we observe
that 
\[
|\fD_{\eps}(\mu^{\epsilon,\gamma})-\fD_{0}(\mu^{\epsilon,\gamma})|\leq\int_{0}^{T}\int_{\Omega}\widetilde{\mathsf{C}}\left(\frac{\sqrt{c_{1}^{\epsilon,\gamma}c_{2}^{\epsilon,\gamma}}}{\eps},b_{2}^{\epsilon,\gamma}\right)\d x\d t\,.
\]
Using the bound from below on $\hat{c}^{\epsilon}$, and the inequality
$\log(x+1)\leq C_{\tilde{p}}x^{\tilde{p}-1}$, we get the estimate
\begin{align*}
\widetilde{\mathsf{C}}\left(\frac{\sqrt{c_{1}^{\epsilon,\gamma}c_{2}^{\epsilon,\gamma}}}{\eps},b_{2}^{\epsilon,\gamma}\right) & \leq\widetilde{\mathsf{C}}\left(C\epsilon^{-\lambda},b_{2}^{\epsilon,\gamma}\right)\leq C\epsilon^{-\lambda}\mathsf{C}\left(C^{-1}\epsilon^{\lambda}b_{2}^{\epsilon,\gamma}\right)\\
 & \leq2C\epsilon^{-\lambda}|b_{2}^{\epsilon,\gamma}|C^{-1}\epsilon^{\lambda}\log\left(C^{-1}\epsilon^{\lambda}|b_{2}^{\epsilon,\gamma}|+1\right)\\
 & \leq2C_{\tilde{p}}|b_{2}^{\epsilon,\gamma}|C^{-\tilde{p}+1}\epsilon^{\lambda(\tilde{p}-1)}|b_{2}^{\epsilon,\gamma}|^{\tilde{p}-1}\leq\tilde{C}|b_{2}^{\epsilon,\gamma}|^{\tilde{p}}\epsilon^{\lambda(\tilde{p}-1)},
\end{align*}
where $\tilde{C}=\tilde{C}(C_{\tilde{p}},C)$. By $\|\dot{c}^{\epsilon}\|_{\L^{\tilde{p}}([0,T]\times\Omega)}\lesssim\frac{1}{\epsilon^{\alpha}}$,
we conclude that $\|\div\hat{J}^{\epsilon}\|_{\tilde{p}}\lesssim\frac{1}{\epsilon^{\alpha}}$,
and by Lemma \ref{lem:IntegrabilityFluxJ}, we have $\hat{J}^{\epsilon}\in\L^{\tilde{p}}([0,T]\times\Omega)$.
Together this implies that $\|b_{2}^{\epsilon}\|_{\L^{\tilde{p}}([0,T]\times\Omega)}\lesssim\frac{1}{\epsilon^{\alpha}}$.
Hence, we get 
\[
|\fD_{\eps}(\mu^{\epsilon,\gamma})-\fD_{0}(\mu^{\epsilon,\gamma})|\lesssim\eps^{\lambda(\tilde{p}-1)}\eps^{-\tilde{p}\alpha}=\eps^{-\frac{1}{2d+1}\left\{ (2d+2)\alpha-\lambda\right\} }.
\]
Choosing, $\lambda,\alpha\in[0,1[$ such that $d+1\leq\frac{\lambda}{2\alpha}$,
we conclude that the right-hand side converges to zero, which proves
the claim.
\end{proof}

\section{Remarks for reaction-diffusion systems involving more species\label{sec:Multispecies}}

In the last section, we comment on linear reaction-diffusion systems
involving more species. The evolution equation for concentrations
$c\in\R_{\geq0}^{I}$, $I\in\N$ is given by 
\[
\dot{c}=\mathrm{diag}(\delta_{1},\dots,\delta_{I})\Delta c+A^{\eps}c
\]
where $A^{\eps}=A^{S}+\frac{1}{\eps}A^{F}$ is a Markov generator
(preserving positivity and total mass), which consists of a slow part
and a fast part. The main assumption is the $A^{\eps}$satisfies detailed
balance with respect to its stationary measure $w^{\eps}$. Similar
to \cite{MieSte19ECLRS,MiPeSt20EDPCNLRS}, we are going to assume
that the stationary vector $w^{\eps}$ satisfies $w^{\eps}\to w^{0}$
as $\eps\to0$ and that $w^{0}>0$. The positivity of the limit stationary
measure $w^{0}$ means that in the limit the evolution respects all
concentrations $c_{i}$ and is not degenerate.

The gradient structure is defined on the state space
\begin{align*}
Q=\mathrm{Prob}(\Omega\times\left\{ 1,\cdots,I\right\} ):=\{\mu=(\mu_{1},\dots,\mu_{I})\in\R^{I}:\mu_{i}\in{\cal M}(\Omega),\ \mu_{i}\geq0,~\mu_{i}(\Omega)=1\}.
\end{align*}
The driving energy functional $\cE_{\eps}:X\rightarrow\R_{\infty}$
has the form 
\begin{align*}
\cE_{\eps}(\mu)=\begin{cases}
\int_{\Omega}\sum_{j=1}^{I}E_{B}\left(\frac{c_{j}}{w_{j}^{\eps}}\right)w_{j}^{\eps}\d x, & \mathrm{~if~}\mu=c\,\d x\\
\infty, & \mathrm{otherwise}.
\end{cases} & ,
\end{align*}
and the dual dissipation potential splits into two parts
\begin{align*}
\cR^{*}(\mu,\xi) & =\cR_{\mathrm{diff}}^{*}(\mu,\xi)+\cR_{\mathrm{react}}^{*}(\mu,\xi)\\
\cR_{\mathrm{diff}}^{*}(\mu,\xi) & =\frac{1}{2}\int_{\Omega}\sum_{j=1}^{I}\delta_{j}|\nabla\xi_{j}(x)|^{2}\d\mu_{j},\ \cR_{\mathrm{react},\eps}^{*}(\mu,\xi)=\int_{\Omega}\sum_{i<j}\kappa_{ij}^{\eps}\mathsf{C}^{*}(\xi_{i}(x)-\xi_{j}(x))\,\d\sqrt{\mu_{i}\mu_{j}},
\end{align*}
where $\kappa_{ij}^{\eps}:=A_{ij}^{\eps}\left(\frac{w_{j}^{\eps}}{w_{i}^{\eps}}\right)^{1/2}$.
In particular, the reaction part of the dissipation potential splits
into a fast part and a slow part 
\begin{align*}
\cR_{\mathrm{react},\eps}^{*}(\mu,\xi) & =\cR_{\slow,\eps}^{*}(\mu,\xi)+\frac{1}{\eps}\cR_{\fast,\eps}^{*}(\mu,\xi)\\
\cR_{\mathrm{xy},\eps}^{*}(\mu,\xi) & =\int_{\Omega}\sum_{i<j}\tilde{\kappa}_{ij}^{\eps}\,\mathsf{C}^{*}(\xi_{i}(x)-\xi_{j}(x))\ \d\sqrt{\mu_{i}\mu_{j}},\quad\mathrm{xy}\in\left\{ \slow,\fast\right\} ,
\end{align*}
where $\tilde{\kappa}_{ij}^{\eps}$ are bounded and positive uniformly
in $\eps>0$. In particular, we call a reaction and its flux $b_{ij}$
slow if $A_{ij}^{\eps}=O(1)$ and fast if $A_{ij}^{\eps}=O(\eps^{-1})$.
Due to the detailed balance assumption and by $w^{0}>0$, the distinction
between fast and slow reactions is indeed well-defined.

In the remainder of the section, we briefly explain how to generalize
the proof of the EDP-convergence result also for this situation. Major
differences occur at two stages, namely in 1) deriving compactness
for slow reaction fluxes, 2) proving the limsup-estimate. The reaction
fluxes of the fast reactions are not seen in the limit and have to
be reconstructed in an analogous way as in the 2-species situation.
Firstly, we explain the compactness result and the liminf-estimate.
Secondly, we comment on the limsup-estimate.

\subsection{Compactness for slow reaction fluxes and liminf-estimate}

Here, we comment on proving compactness and the liminf-estimate for
the multi species case. In comparison to the previous situation, compactness
for the concentrations, by using strong compactness for coarse-grained
variables and convergence towards the slow manifold, can be derived
(cf. Lemmas \ref{LemmaRegInTime}, \ref{lem:ConvergenceToSlowManifoldAndStongCompactness}).
Moreover, compactness of diffusion fluxes and spatial regularity follows,
too (cf. Lemma \ref{lem:IntegrabilityFluxJ}).

In contrast to the situation of two species connected with one fast
reaction, where no slow reaction fluxes exists, compactness for slow
reaction fluxes $b_{ij}^{\eps}$ has to be derived in the multi species
case. This follows immediately from Lemma \ref{lem:LiminfEstimateAGS},
once compactness of $\sqrt{c_{i}^{\eps}c_{j}^{\eps}}$ is obtained.
At this point it is clear that weak convergence of $c^{\eps}\rightharpoonup c^{0}$
is not sufficient. Instead the previously derived strong convergence
of $c^{\eps}\rightarrow c^{0}$ implies by dominated convergence also
strong convergence of $\sqrt{c_{i}^{\eps}c_{j}^{\eps}}\rightarrow\sqrt{c_{i}^{0}c_{j}^{0}}$
in $\L^{1}([0,T]\times\Omega)$, and hence, compactness for the slow
fluxes $b_{ij}^{\eps}$. Compactness for fast reaction fluxes can
not be obtained as already mentioned in Remark \ref{rem:NoBoundOnFastReactionFluxes}.
Having proved compactness, the proof of the liminf-estimate is exactly
the same as for Theorem \ref{thm:LiminfEstimate}, since the functional
$\cD_{\eps}$ is jointly convex in all variables $(c,J,b)$.

\subsection{Equilibration and reconstruction of reaction fluxes und recovery
sequence}

A crucial observation throughout the proof of the $\Gamma$-convergence
was Lemma \ref{lem:Convexity}, which provides an equilibration of
fluxes assuming microscopic equilibria for the concentrations. In
Lemma \ref{prop:DissipationFunctionalInCoarseGrainedVariables}, we
derived equilibration for the diffusion fluxes. Similarly, also an
equilibration of the slow reaction fluxes can be derived. In \cite{MieSte19ECLRS}
a general operator-theoretic coarse-graining and reconstruction procedure
has been developed. This method can also be applied to derive coarse-grained
fluxes and a coarse-grained continuity equation, see \cite{Step21CGRCLRS}.
Importantly for us, as in \eqref{eq:Reconstruction} the reconstructed
slow reaction fluxes depend linearly on the coarse-grained reaction
fluxes. The fast reaction fluxes are then of the form
\[
b_{ij}=a_{1}\div\hat{J}_{i}+a_{2}\hat{J}_{i}+\sum_{j=3}^{k(i)}a_{ij}\hat{b}_{ij},
\]
where all functions $a_{j}$ are $\mathrm{C}^{0}(\Omega,\R^{k_{j}})$,
where $k_{1}=1$ , $k_{2}=d$ and $k_{ij}=1$.

In order to prove that the constant sequence for smooth and positive
concentrations is indeed a recovery sequence, we follow the same reasoning
as in the Lemmas \ref{lem:BinLCimpliesCepstozero} and \ref{prop:ConvOfCTermWithSmoothing}.
The only difference comes from the explicit depends on the coarse-grained
reaction flux $\hat{b}_{ij}$, Using the bound on the limit dissipation
functional (which provides bounds on $\int_{0}^{T}\!\!\int_{\Omega}\widetilde{\mathsf{C}}(\sqrt{\hat{c}_{i}\hat{c}_{j}},\hat{b}_{ij})\d x\d t$)
and the next Lemma \ref{lem:InequalityForC}, we obtain that $\hat{b}_{ij}\in\L^{\mathsf{C}}([0,T]\times\Omega)$
(we refer to \cite{FreMie21DTL} for a proof). Since $\L^{\mathsf{C}}$
is an Orlicz-space, we conclude that the reconstructed fluxes $b_{i}$
are in $\L^{\mathsf{C}}$. This allows to proceed as in Lemma \ref{prop:ConvOfCTermWithSmoothing}
and proves the existence of a recovery sequence.
\begin{lem}[\cite{FreMie21DTL}]
\label{lem:InequalityForC}Let $p>1$. Then, for all $a\geq0$ and
$B\in\R$ we have
\[
\widetilde{\mathsf{C}}(a,B)\geq\left(1-\frac{1}{p}\right)\widetilde{\mathsf{C}}(B)-\frac{2}{p}a^{p}.
\]
In particular, setting $a=\sqrt{c_{i}c_{j}}$ and $B=b_{ij}$ we have
\[
\int_{0}^{T}\!\!\int_{\Omega}\widetilde{\mathsf{C}}(\sqrt{c_{i}c_{j}},b_{ij})\d x\d t+\|\sqrt{c_{i}c_{j}}\|_{\L^{p}([0,T]\times\Omega)}^{p}\gtrsim\int_{0}^{T}\!\!\int_{\Omega}\widetilde{\mathsf{C}}(b_{ij})\d x\d t,
\]
which proves that $b_{ij}\in\L^{\mathsf{C}}([0,T]\times\Omega)$ if
$c_{i},c_{j}\in\L^{p}([0,T]\times\Omega)$ for some $p>1$.
\end{lem}

\vspace{1.5cm}\noindent \textbf{Acknowledgement:} The research was
supported by Deutsche Forschungsgemeinschaft (DFG) through the Collaborative
Research Center SFB 1114 ``\textit{Scaling Cascades in Complex Systems}''
(Project no. 235221301), subproject C05 ``Effective models for materials
and interfaces with multiple scales''. The author is grateful to
Alexander Mielke for many helpful discussions and for having proposed
the problem.

{\footnotesize{}\bibliographystyle{/Users/astephan/Documents/Science/my_alpha}
\bibliography{/Users/astephan/Documents/Science/total-bib-file}
}{\footnotesize\par}
\end{document}